\documentclass{article}

\usepackage{amsmath,amssymb,amsthm}
\usepackage[pdftex]{color}
\usepackage[bookmarks=true,hyperindex,pdftex,colorlinks, citecolor=MidnightBlue,linkcolor=MidnightBlue, urlcolor=MidnightBlue]{hyperref}
\usepackage[dvipsnames]{xcolor}
\usepackage{enumerate}
\usepackage{lineno}
\usepackage[edges]{forest}
\usepackage{mathtools}
\usepackage{graphicx}
\usepackage{subfig}

\usetikzlibrary{decorations.markings}
\tikzset{neg/.style={
		decoration={markings,
			mark= at position 0.5 with {
				\node[transform shape] (tempnode) {$\setminus$};
			}
		},
		postaction={decorate}
}}

\numberwithin{equation}{section}

\newcommand{\MAT}{\left[ \begin{array}}  
	\newcommand{\mat}{\end{array} \right]}

\newtheorem{theorem}{Theorem}[section]

\newtheorem{corollary}[theorem]{Corollary}
\newtheorem{lemma}[theorem]{Lemma}

\theoremstyle{definition}

\newtheorem*{definition*}{Definition}
\newtheorem*{proposition*}{Proposition}
\newtheorem*{theorem*}{Theorem}
\newtheorem*{corollary*}{Corollary}
\newtheorem*{example*}{Example}
\newtheorem*{problem*}{Problem}

\theoremstyle{remark}
\newtheorem{remark}[theorem]{Remark}

\def \st {\operatorname*{s.t. }}

\def \R {\mathbb{R}}

\def \x {{\bf x}}

\def \y {{\bf y}}

\newcommand*{\stcomp}[1]{{#1}^{\mathsf{c}}}

\newcommand{\interior}[1]{%
	{\kern0pt#1}^{\mathrm{o}}%
}

\title{On Continuous Terminal Embeddings of Sets of Positive Reach}

\author{%
  Simone Brugiapaglia$^1$, Rafael Chiclana$^2$, Tim Hoheisel$^3$, Mark Iwen$^{4}$\\
   \\
   {\small $^1$Department of Mathematics and Statistics, Concordia University, \href{mailto:simone.brugiapaglia@concordia.ca}{simone.brugiapaglia@concordia.ca}.}\\
     { \small $^2$Department of Mathematics, Michigan State University, \href{mailto:chiclan1@msu.edu}{chiclan1@msu.edu}.}\\
 {\small $^3$Department of Mathematics and Statistics, McGill University, \href{mailto:tim.hoheisel@mcgill.ca}{tim.hoheisel@mcgill.ca}.}\\
 {\small $^4$Department of Mathematics, and Department of Computational Mathematics, Science}\\ {\small and Engineering (CMSE), Michigan State University, \href{mailto:iwenmark@msu.edu}{iwenmark@msu.edu}.}
}

\begin{document}

\maketitle

\begin{abstract}
    In this paper we prove the existence of H\"{o}lder continuous terminal embeddings of any desired $X \subseteq \mathbb{R}^d$ into $\mathbb{R}^{m}$ with $m=\mathcal{O}(\varepsilon^{-2}\omega(S_X)^2)$, for arbitrarily small distortion $\varepsilon$, where $\omega(S_X)$ denotes the Gaussian width of the unit secants of $X$. More specifically, when $X$ is a finite set we provide terminal embeddings that are locally $\frac{1}{2}$-H\"{o}lder almost everywhere, and when $X$ is infinite with positive reach we give terminal embeddings that are locally $\frac{1}{4}$-H\"{o}lder everywhere sufficiently close to $X$ (i.e., within all tubes around $X$ of radius less than $X$'s reach). When $X$ is a compact $d$-dimensional submanifold of $\mathbb{R}^N$, an application of our main results provides terminal embeddings into $\tilde{\mathcal{O}}(d)$-dimensional space that are locally H\"{o}lder everywhere sufficiently close to the manifold. 
\end{abstract}

	\section{Introduction}\label{sec 1}
	

Let $X \subseteq \mathbb{R}^d$ and $\varepsilon \in (0,1)$. A bi-Lipschitz function $f: X \rightarrow \mathbb{R}^m$ satisfying 
\begin{equation}
(1 - \varepsilon) \| { x}- { y} \| \leq \| f( { x}) - f({ y}) \| \leq (1 + \varepsilon) \| { x}- { y} \| \quad \forall \, x, y \in X
\label{equ:JLembedding}
\end{equation}
is called an {\it (Euclidean) embedding of $X$ into $\mathbb{R}^m$ with distortion $\varepsilon$}, where $m$ is referred to as the {\it embedding dimension}. Over the last two decades such embeddings have been considered/utilized in tens of thousand of publications in areas ranging from computer science (where $X$ is often a finite point set) to randomized numerical linear algebra (where $X$ is often a low-dimensional subspace) to engineering applications (where $X$ is often the set of all vectors in $\mathbb{R}^d$ having at most $s \ll d$ nonzero entries).  In all of these areas the popularity of embeddings with distortion $\varepsilon$ is largely due to the Johnson-Lindenstrauss (JL) lemma \cite{johnson1984extensions} and its subsequent simplifications (see, e.g., \cite{DG2003Elementary,A2001Database}) which, together with the later development of compressive sensing 
\cite{FH2013Mathematical}, showed that random matrices can easily provide such embeddings for $m$ very small.  For example, if one simply sets $f({ x}) = \Pi { x}$ where $\Pi \in \mathbb{R}^{m \times d}$ has suitably normalized and independent subgaussian (e.g., Gaussian) entries, then $f$ will satisfy \eqref{equ:JLembedding} for an arbitrary finite set $X$ with high probability provided that $m \geq c \log(|X|)/ \varepsilon^2$, where $c>0$ is a universal constant \cite{vershynin2018high-dimensional}. 

Though it is certainly fantastic that a random linear embedding $f = \Pi$ can be used to satisfy  \eqref{equ:JLembedding}, human nature led researchers to ask for even more almost immediately.  One direction of inquiry involved exploring whether the embedding dimensions $m$ achievable by random matrices might be further reduced by using other (e.g., nonlinear) embedding functions $f$ instead.  This leads to several interesting (and still not entirely resolved) questions, including:  What is the smallest achievable embedding dimension $m$ such that $\exists\, f: X \rightarrow \mathbb{R}^m$ satisfying \eqref{equ:JLembedding} for a given $X \subseteq \mathbb{R}^d$ of interest? Given $X \subseteq \mathbb{R}^d$ and $\varepsilon \in (0,1)$, how does the smallest achievable embedding dimension $m$ attainable by a {\it linear} embedding $f = \Pi$, where $\Pi \in \mathbb{R}^{m \times d}$, compare to the smallest embedding dimension achievable by {\it any} (e.g., potentially nonlinear) function $f$?  And, more specifically, for what $X \subseteq \mathbb{R}^m$ and $\varepsilon \in (0,1)$ pairs does the Johnson–Lindenstrauss lemma nearly achieve the smallest possible embedding dimension obtainable by any $f: X \rightarrow \mathbb{R}^m$ with high probability (w.h.p.) using a {\it random linear} embedding?

Work on answering these questions includes, e.g., 
\cite{larsen2017optimality} in which Larson and Nelson construct a general class of finite sets $X \subseteq \mathbb{R}^d$ such that {\it any} function $f: X \rightarrow \mathbb{R}^m$ satisfying \eqref{equ:JLembedding} must have $m > c \log(|X|) / \varepsilon^2$ for all $\varepsilon$ not too small.  When combined with known upper bounds on $m$ achievable via random matrices this establishes the JL lemma as ``near-optimal'' even when compared to best-possible nonlinear embeddings, at least for some worst-case finite sets.  Other work in this area showed that random matrices are also near-optimal w.h.p.\ compared to the best possible {\it linear} embeddings $f = \Pi$ satisfying \eqref{equ:JLembedding} when $X$ is chosen to be the set of all $s$-sparse vectors in $\mathbb{R}^d$ \cite{FH2013Mathematical}.  In \cite{iwen2023lower} these prior results where then generalized by noting that any (even nonlinear) $f: X \rightarrow \mathbb{R}^m$ satisfying \eqref{equ:JLembedding} must have $m$ scale like the squared Gaussian
width of $X$ for all $X \subseteq \mathbb{R}^d$, after which specific choices of $X$ both $(i)$ reproduce \cite{larsen2017optimality} for some ranges of $\varepsilon$ and $(ii)$ show that w.h.p.\ random matrices actually embed with distortion $\varepsilon$ all $s$-sparse vectors in $\mathbb{R}^d$ into $\mathbb{R}^m$ with an embedding dimension $m$ that is nearly as small as any linear {\it or nonlinear} function $f$ can achieve.  In addition, \cite{iwen2023lower} also demonstrates the existence of a general class of low-dimensional submanifolds of $\mathbb{R}^d$ which random matrices embed with distortion $\varepsilon$ w.h.p.\ into $\mathbb{R}^m$ with $m$ nearly as small as any $f$ can achieve (even a nonlinear one).  All together, \cite{iwen2023lower} thereby shows that {\it linear} embeddings perform as well as nonlinear ones in a large range of interesting situations, thereby demonstrating that restricting $f$ to be linear is often much less limiting with respect to minimizing embedding dimensions $m$ than one might initially expect.  In short, it appears as if there is often surprisingly little to gain with respect to embeddings with distortion $\varepsilon$ by allowing $f$ to be nonlinear.
 

In light of the previous discussion one might reasonably ask what undeniable benefits a nonlinear embedding can provide that a linear one simply cannot.  So-called terminal embeddings, first proposed in the computer science literature by Elkin, Filtser, and Neiman \cite{elkin2017terminal}, provide an answer.  Let $X \subseteq \mathbb{R}^d$ and $\varepsilon \in (0,1)$. A  function $f: \mathbb{R}^d \rightarrow \mathbb{R}^m$ satisfying 
\begin{equation}\label{eq: terminal condition} (1-\varepsilon)\|x-y\| \leq \|f(x)-f(y)\| \leq (1+\varepsilon)\|x-y\| \quad \forall \, x \in X \quad \forall \, y \in \mathbb{R}^d
\end{equation}
is called an {\it terminal embedding of $X$ into $\mathbb{R}^m$ with distortion $\varepsilon$}, where $m$ is again referred to as the {\it embedding dimension}.  {\bf Note here -- crucially -- that \eqref{eq: terminal condition} must hold for all} $\boldsymbol y \in\mathbb{R}^{\boldsymbol d}$.  It is a simple exercise to see that any $f$ satisfying this condition for any nonempty $X$ cannot possibly be linear unless $m = d$ (see, e.g., the remark following \cite[Theorem 1.2]{narayanan2019optimal}).  Nonetheless, a series of works \cite{elkin2017terminal,mahabadi2018nonlinear,narayanan2019optimal} somewhat recently established the existence of such nonlinear functions $f$ for arbitrary {\it finite} sets $X$.  Better yet, they can still achieve near-optimal embedding dimensions $m = c \log(|X|)/\varepsilon^2$ \cite{narayanan2019optimal}.  Even more recent work in this direction \cite{chiclana2024on} further generalized these terminal embedding results by showing that, in fact, terminal embeddings with near-optimal embedding dimensions also exist for arbitrary (e.g., infinite) subsets $X \subseteq \mathbb{R}^d$, with special attention given to the case where $X$ is a low-dimensionsal submanifold of $\mathbb{R}^d$.  In addition, \cite{chiclana2024on} also explores potential data science applications of these embeddings, empirically demonstrating, e.g., that they can significantly outperform standard JL-embeddings in compressive classification tasks.  

As mentioned above, one cannot expect terminal embeddings to be linear. A natural question is whether they can nevertheless be constructed with some form of global regularity (e.g., continuity). This is not immediate, as existing constructions achieving near-optimal embedding dimensions \cite{narayanan2019optimal,chiclana2024on} are defined point-by-point via optimization, a process that does not, in general, preserve regularity. In this work, we show that a suitable modification of this approach yields terminal embeddings with near-optimal embedding dimensions that satisfy H\"older continuity.

\subsection{Motivation for regularity}

While terminal embeddings are primarily designed to preserve distances to a fixed dataset $X$, additional regularity properties can play an important role both theoretically and algorithmically. We highlight several motivations for studying regular terminal embeddings:

\begin{itemize}
	\item \textbf{Algorithmic efficiency.}
	The construction of terminal embeddings is inherently pointwise and typically involves solving an optimization problem for each query point. Continuity suggests that nearby points have similar feasible sets and therefore similar minimizers. As a result, the solution corresponding to a point $u$ can be used to initialize the optimization problem at a nearby point $v$, leading to effective warm-start strategies. This is particularly relevant in iterative or streaming settings, where queries are processed sequentially and tend to be correlated. In such cases, continuity can significantly reduce the number of iterations required for convergence and improve overall computational efficiency.
	
	\item \textbf{Global construction via interpolation.}
	Continuity allows one to compute the embedding on a discrete net and extend it to the entire space via interpolation. More precisely, one can evaluate the embedding on a sufficiently fine $\delta$-net and use interpolation schemes (e.g., piecewise linear or Lipschitz extensions) to define $f$ elsewhere. The regularity of the map ensures that the interpolation error can be controlled in terms of $\delta$, yielding a global construction from finitely many evaluations. This provides a practical way to approximate the embedding over large domains without solving the underlying optimization problem at every point.
	
	\item \textbf{Stability and robustness.}
	Continuity guarantees that small perturbations in the input $u$ lead to small changes in $f(u)$. This is particularly relevant in applications where data is noisy or subject to numerical errors, as discontinuities can lead to unstable behavior. Moreover, quantitative regularity such as H\"older continuity provides explicit control on how rapidly the embedding can change, ensuring that small perturbations in the input do not lead to abrupt variations in the output.
	
	\item \textbf{Compatibility with geometric structure.}
	In many settings of interest, the dataset $X$ has additional geometric properties (e.g., manifold structure). Continuity ensures that the embedding extends in a coherent way to neighborhoods of $X$, preserving local geometric features rather than behaving in a purely pointwise manner.
\end{itemize}
 
\subsection{Contributions}
	
In this work, we exploit previous construction techniques to produce terminal embeddings with desirable regularity properties, namely H\"{o}lder continuity. Our first main result considers \emph{finite} sets and provides optimal dimensionality reduction through terminal embeddings that are locally $\frac{1}{2}$-H\"{o}lder continuous almost everywhere.
	
	\begin{theorem}[Finite case]\label{theo: main finite}
		Let  $X=\{x_1,\ldots,x_n\}\subseteq\mathbb{R}^d$ and $\varepsilon \in (0,1)$. Then, there exists   $m=\mathcal{O}(\varepsilon^{-2}\log n)$ and a terminal embedding $f\colon \mathbb{R}^d \longrightarrow \mathbb{R}^{m+1}$ with distortion  $\varepsilon$ such that $f$ is locally $\frac{1}{2}$-H\"{o}lder almost everywhere, i.e.,  for almost every $u \in \mathbb{R}^d$, there exists a neighborhood $U$ of $u$ and $D_u>0$ such that
		\[ \|f(v)-f(w)\| \leq D_u \|v-w\|^{1/2} \quad \forall \, v, w \in U.\]
	\end{theorem}

	\begin{remark}
		Regularity in Theorem \ref{theo: main finite} is achieved for any point $u \in \mathbb{R}^d \setminus X$ for which there is a unique orthogonal projection to $X$. In other words, if $u \in \mathbb{R}^d \setminus X$ has a unique nearest neighbor in $X$, then the local $\frac{1}{2}$-H\"{o}lder regularity holds at $u$. 
		On the other hand, at points where the nearest neighbor is not unique, our construction does not, in general, guarantee continuity. This reflects the fact that the embedding is defined through a choice of representative point in $X$, which may change discontinuously when the nearest neighbor is not unique.
	\end{remark}

	
	Given a set $X \subseteq \mathbb{R}^d$, the {\em reach} \cite{federer1959curvature}  of $X$ measures how far away points can be from $X$, while still having a unique closest point in its closure $\overline{X}$ . Formally, the \textit{reach} of $X$, denoted $\tau_X$, is defined as 
	\begin{equation}
        \label{eq:def_reach}
	\tau_X := \sup \left\{ t \geq 0  \colon \, \forall \, x \in \mathbb{R}^d \text{ such that } d(x,X) < t, \, x \text{ has a unique closest point in } \overline{X} \right\},
	\end{equation}
	where $d(x,X)$ denotes the Euclidean distance of a point $x\in \mathbb{R}^d$ from a set $X \subseteq \mathbb{R}^d$. See Figure~\ref{fig:reach} for an illustration.
 \begin{figure}
     \centering
     \includegraphics[width=0.4\linewidth]{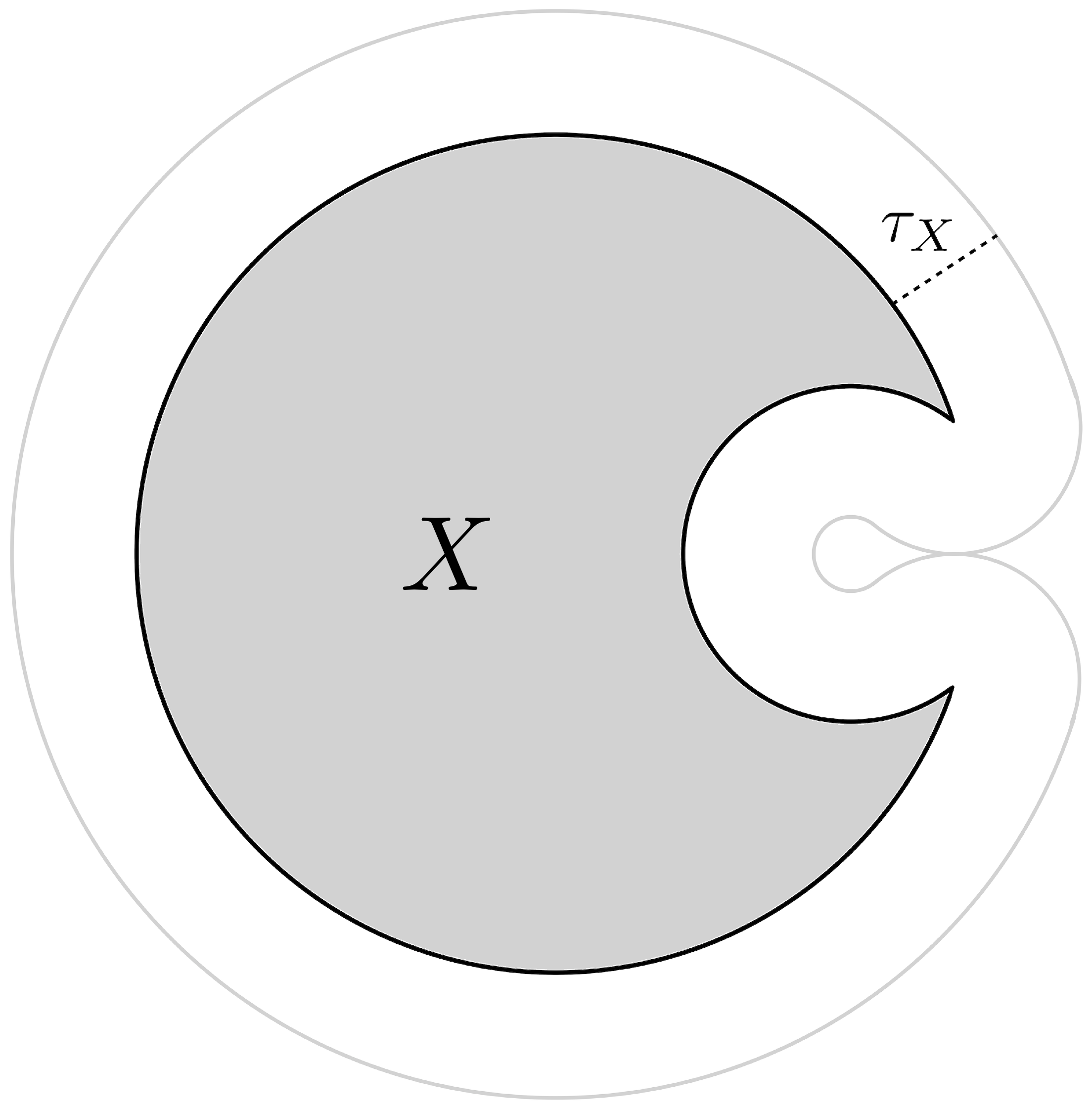}
     \caption{A nonconvex set $X \subseteq \mathbb{R}^2$ and its reach $\tau_X$, defined in \eqref{eq:def_reach}.}
     \label{fig:reach}
 \end{figure}
 We note that the sets in $\R^d$  with reach $+\infty$ are exactly the convex sets see, e.g., \cite[Corollary 21.13]{BaC11}.

	For infinite sets $X$, the optimal embedding dimension of a terminal embedding is determined by the Gaussian width of the \textit{unit secants of $X$}, which is the set defined by
	\begin{equation}\label{eq: S_X} S_X=\overline{\left \{\frac{x-y}{\|x-y\|} \colon x \neq y \in X \right \}}\subseteq S^{d-1},
	\end{equation}
	where $S^{d-1}$ denotes the unit sphere of $\mathbb{R}^d$. 
	Recall that the \textit{Gaussian width} of $S_X$ is  
	\begin{align}\label{eq:GWidth}
		w(S_X) := \mathbb{E} \sup_{x \in S_X} \, \langle g,x \rangle,
	\end{align}
	where $g$ is a random vector with $d$ independent entries following a standard normal distribution. The Gaussian width is considered one of the basic geometric quantities associated with subsets $X$ of $\mathbb{R}^d$, such as volume and surface area. Indeed, this quantity plays a central role in high-dimensional probability and its applications. For a nice introduction to the notion of Gaussian width we refer the interested reader to \cite[Chapter 7]{vershynin2018high-dimensional}.
	
	Our second main result holds for any (possibly, infinite) subset $X$ of $\mathbb{R}^d$ with positive reach $\tau_X>0$, and provides optimal dimensionality reduction through terminal embeddings that are $\frac{1}{4}$-H\"{o}lder continuous for every point within the reach of $X$. Given $x \in \mathbb{R}^d$, and $r>0$, let $B(x,r)$ denote the open ball with center $x$ and radius $r$.

	\begin{theorem}[General case]\label{theo: main infinite}
		Let $X\subseteq\mathbb{R}^d$ with reach $\tau_X>0$ and $\varepsilon \in (0,1)$. Then there exists $m=\mathcal{O}(\varepsilon^{-2}\omega(S_X)^2)$ and a terminal embedding $f\colon \mathbb{R}^d \longrightarrow \mathbb{R}^{m+1}$ with distortion $\varepsilon$ such that $f$ is locally $\frac{1}{4}$-H\"{o}lder on $X+B(0,\tau_X/2)$, i.e., for every $u \in \mathbb{R}^d$ with $d(u,X)<\tau_X/2$ there is a neighborhood $U$ of $u$ and  $C_u>0$ such that
		\[ \|f(v)-f(w)\| \leq C_u \|v-w\|^{1/4} \quad \forall \, v, w \in U,\]
		where $C_u>0$ is a constant that only depends on the distance between $u$ and the set $X$ and on $\varepsilon$.
	\end{theorem}
	
	It is worth mentioning that Theorem \ref{theo: main finite} is not a special case of Theorem \ref{theo: main infinite}. Indeed, in the finite case Theorem \ref{theo: main finite} provides stronger regularity on a larger domain than Theorem \ref{theo: main infinite}.
	
	\begin{remark}
		The proofs of Theorem \ref{theo: main finite} and Theorem \ref{theo: main infinite} are constructive, yielding algorithms that can be used to generate terminal embeddings in practical settings. Furthermore, we obtain explicit expressions of both the neighborhood $U$ and the constant $C_u>0$ in Theorem \ref{theo: main infinite}. If $U'$ and $C'_u$ are the neighborhood and the constant discussed in Remark \ref{rem: constants theo holder}, then one can take $U=U'$ and 
		\[ C_u = \sqrt{32+2C'^{2}_u + (6+2C'_u)(d(u,X) + 2)}.\]
	\end{remark}

    \begin{remark} \label{rem:UptoReachOK}
        One can show that the terminal embedding $f$ guaranteed by Theorem~\ref{theo: main infinite} is in fact 
        locally $\frac{1}{4}$-H\"{o}lder on $X+B(0,\tau_X/a)$ for any $a > 1$.  Herein we chose $a = 2$ for simplicity.
    \end{remark}

    As an example application of Theorem~\ref{theo: main infinite} we may now demonstrate the existence of terminal manifold embeddings which are continuous everywhere sufficiently close to the manifold one is embedding.  First, however, we state a bound on the Gaussian width of the unit secants of a smooth $d$-dimensional submanifold of $\mathbb{R}^N$ in terms of the manifold's dimension, reach, and volume.  The following is a restatement of Theorem 4.5 in \cite{iwen2024on}.

\begin{lemma}\label{GaussianWidthOfManifodWithBoundaryViaGunther}
	Let $\mathcal{M}$ be a compact $d$-dimensional submanifold of $\mathbb{R}^N$ 
	with boundary $\partial \mathcal{M}$, finite reach $\tau_{\mathcal{M}}$, and volume  $V_{\mathcal M}$.  Enumerate the connected components of $\partial \mathcal{M}$ and let $\tau_i$ be the reach of the $i^{\rm th}$ connected component of $\partial \mathcal M$ as a submanifold of $\mathbb{R}^N$. 
	Set $\tau := \min_{i} \{\tau_{\mathcal M}, \tau_i \}$, let $V_{\partial \mathcal{M}}$ be the volume of $\partial \mathcal{M}$, and denote the volume of the $d$-dimensional Euclidean ball of radius $1$ by $\omega_d$.  Next, 
	\begin{enumerate}
		\item if $d=1$, define $\alpha_{\mathcal M} :=  \frac{20 V_{\mathcal M}}{\tau}  + V_{\partial {\mathcal M}}$, else
		\item if $d \geq 2$, define $\alpha_{\mathcal M} :=
		\frac{V_\mathcal{M}}{ \omega_d} \left(\frac{41}{\tau} \right)^d 
		+ \frac{V_{\partial \mathcal{M}}}{ \omega_{d-1}} \left(\frac{81}{\tau} \right)^{d-1}$.  
	\end{enumerate}
	Finally, define 
	\begin{align}
		\beta_{\mathcal{M}} &:= \left(\alpha_{\mathcal M}^2 +3^{d} \alpha_{\mathcal M} \right). \label{equ:betadef}
	\end{align}
	Then, the Gaussian width of the unit secants of $\mathcal{M}$ satisfies $$w \left(  S_{\mathcal M} \right)  \leq 8\sqrt{2}\sqrt{\log \left(\beta_{\mathcal{M}} \right)+4d}.$$
\end{lemma}

Given Lemma~\ref{GaussianWidthOfManifodWithBoundaryViaGunther} the following result is a direct consequence of Theorem \ref{theo: main infinite} and Remark~\ref{rem:UptoReachOK}.

\begin{corollary}\label{cor:main}
	Let $\mathcal{M}$ be a compact $d$-dimensional submanifold of $\mathbb{R}^N$ 
	with boundary $\partial \mathcal{M}$, finite reach $\tau_{\mathcal{M}}$, and volume  $V_{\mathcal M}$.  Enumerate the connected components of $\partial \mathcal{M}$ and let $\tau_i$ be the reach of the $i^{\rm th}$ connected component of $\partial \mathcal M$ as a submanifold of $\mathbb{R}^N$. 
	Set $\tau := \min_{i} \{\tau_{\mathcal M}, \tau_i \}$, let $V_{\partial \mathcal{M}}$ be the volume of $\partial \mathcal{M}$, and denote the volume of the $d$-dimensional Euclidean ball of radius $1$ by $\omega_d$. Set $\beta_{\mathcal{M}}$ as in (\ref{equ:betadef}). Then, for any $\varepsilon \in (0,1)$ there exists $f \colon \mathbb{R}^N \longrightarrow \mathbb{R}^m$ with $m=\mathcal{O}(\varepsilon^{-2}(\log (\beta_{\mathcal{M}}) + d))$ such that both
    \begin{align*}
	{\it (i)}&~~(1-\varepsilon)\|\x-\y\|_2 \leq \|f(\x)-f(\y)\|_2 \leq (1+\varepsilon)\|\x-\y\|_2 \quad \forall \, \x \in \mathcal{M} \quad \forall \, \y \in \mathbb{R}^N,~{\rm and}\\
    {\it (ii)}&~~f \textrm{~is~locally~} \tfrac{1}{4}\textrm{-H\"{o}lder~on~} {\mathcal M}+B(0,r) \quad \forall \, r \in (0,\tau_{\mathcal M}).
    \end{align*}
\end{corollary}

We are now prepared to discuss preliminary results.


	\subsection{Previous work}\label{sec 2}
	
	In this subsection, we summarize step by step the process used in \cite{mahabadi2018nonlinear}, \cite{narayanan2019optimal}, and \cite{chiclana2024on} to construct terminal embeddings with arbitrarily small distortion. Then, we show how to exploit these constructions to generate terminal embeddings with  desired regularity properties.

	Let $X$ be a subset of $\mathbb{R}^d$. The first step in constructing a terminal embedding for $X$ is to find a linear Johnson-Lindenstrauss embedding $\Pi \colon \mathbb{R}^d \longrightarrow \mathbb{R}^m$ of $X$ with small distortion. The strategy is then to extend this mapping beyond $X$ in a strategically designed manner to satisfy the terminal condition (\ref{eq: terminal condition}). To achieve optimal embedding dimension, we will use a notion of embedding that is slightly stronger than the one provided by the Johnson-Lindenstrauss lemma. Let $T$ be a subset of the unit sphere of $\mathbb{R}^d$, and $\varepsilon \in (0,1)$. A linear mapping $\Pi \colon \mathbb{R}^d \longrightarrow \mathbb{R}^m$ is said to provide \textit{$\varepsilon$-convex hull distortion} for $T \subseteq \R^d$ if
	\[ \big{|}\|\Pi x\|-\|x\|\big{|}<\varepsilon \quad \forall \, x \in \operatorname{conv}(T),\]
	where $\operatorname{conv}(T)$ denotes the \textit{convex hull of $T$}. In this work, we typically use the set $T$ above to represent the unit secants of $X$, which is the set $S_X$ defined in (\ref{eq: S_X}).
	
	The existence of embeddings $\Pi \colon \mathbb{R}^d \longrightarrow \mathbb{R}^m$ providing $\varepsilon$-convex hull distortion for a set $X\subseteq \mathbb{R}^d$, for arbitrarily small $\varepsilon\in(0,1)$, has been previously studied. For finite sets $X$ with $n$ elements, \cite[Corollary 3.5]{narayanan2019optimal} shows that such mappings exist for $m=O(\varepsilon^{-2} \log(n))$. This follows as an application of powerful results from \cite{dirksen2016dimensionality}. For arbitrary sets $X$,  \cite[Corollary 3.2]{chiclana2024on} provides existence for $m=O(\varepsilon^{-2} \omega(S_X)^2)$, where $\omega(S_X)$ is the Gaussian width of $S_X$, see \eqref{eq:GWidth}. This was obtained as an application of the matrix deviation inequality (see \cite[Theorem 9.1.1]{vershynin2018high-dimensional}).
	
	Let $X$ be a subset of $\mathbb{R}^d$ and $\varepsilon \in (0,1)$. For any $u \in \mathbb{R}^d$, we select a point $u_{NN}$ from its closure $\overline{X}$  that minimizes the distance to $u$, i.e., 
	\begin{equation}\label{eq: u_NN} 
		u_{NN} \in \arg\min_{x \in \overline{X}} \; \| u-x\|.
	\end{equation}
	\begin{remark}\label{rem: u_NN unique}
		The subscript $NN$ alludes to  {\em nearest neighbor}, a terminology commonly used in computer science. Clearly, the mapping $u \mapsto u_{NN}$ is defined for any $u \in \mathbb{R}^d$, as it simply picks (according to some rule) an element from the (Euclidean) projection of $u$ onto $\overline{X}$, which becomes unique when $u$ is within reach of $X$.
	\end{remark}
	Consider an embedding $\Pi \colon \mathbb{R}^d \longrightarrow \mathbb{R}^m$ providing $\frac{\varepsilon}{6}$-convex hull distortion for $S_X$. The terminal embeddings presented in previous works (namely, \cite{mahabadi2018nonlinear}, \cite{narayanan2019optimal}, and \cite{chiclana2024on}) are constructed by extending $\Pi$ beyond $X$. More precisely, they define $f \colon \mathbb{R}^d \longrightarrow \mathbb{R}^{m+1}$ by
	\begin{equation}\label{eq: terminal embedding}
		f(u)=\left (\Pi u_{NN} +u',\sqrt{\|u-u_{NN}\|^2 - \|u'\|^2}\right ) \quad \forall \, u \in \mathbb{R}^d,
	\end{equation}
	where $u' \in \mathbb{R}^m$ is a point with $\|u'\|\leq \|u-u_{NN}\|$ that satisfies the constraints
	\[ \big{|}\langle u',\Pi(x-u_{NN})\rangle - \langle u-u_{NN}, x-u_{NN}\rangle \big{|} \leq \varepsilon \|u-u_{NN}\|\|x-u_{NN}\| \quad \forall \, x \in X.\]
	Later in the paper, the point $u'$ will be denoted by $\alpha(u)$ when working with finite sets, and by $\beta(u)$ when working with infinite sets. Observe that if $u \in X$, then $u_{NN}=u$. Thus, from $\|u'\|\leq \|u-u_{NN}\|$ we deduce $u'=0$. Therefore, $f$ can be seen as an extension of $\Pi$  in the following sense: we have $f(u)=(\Pi u,0)$ for all $u \in X$. Furthermore, the above constraints on $u'$ guarantee that such an extension satisfies the terminal condition (\ref{eq: terminal condition}). Intuitively, $u'$ is a point that approximately preserves all the angles formed when we consider vectors of the form $x-u_{NN}$ with $x \in X$. The preservation of these angles will then lead to preservation of the distances. It is important to note that the existence of such a point $u'$ is far from trivial, and it was first obtained for finite sets $X$ in \cite{mahabadi2018nonlinear} using the von Neumann's minimax theorem \cite{neumann1928zur}. Later on, it was optimized by \cite[Lemma 3.6]{narayanan2019optimal}, which states:
	
	\begin{lemma}[{\cite[Lemma 3.6]{narayanan2019optimal}}]\label{lem: narayan u'}
		Let $x_1,\ldots,x_n\in\mathbb{R}^d\setminus \{0\}$. Suppose that $\Pi \in \mathbb{R}^{m \times d}$ provides $\varepsilon$-convex hull distortion for $V=\left \{\pm \frac{x_i}{\|x_i\|}\colon i=1,\ldots,n\right \}$. Then, for any $u \in \mathbb{R}^d$, there is $u' \in \mathbb{R}^m$ such that $\|u'\|\leq \|u\|$ and $|\langle u',\Pi x_i\rangle - \langle u,x_i\rangle|\leq \varepsilon \|u\|\|x_i\|$ for every $x_i$.
	\end{lemma}
	
	The previous result was generalized in \cite{chiclana2024on} for arbitrary subsets of $\mathbb{R}^d$.
	
	\begin{lemma}[{\cite[Lemma 3.4]{chiclana2024on}}]\label{lem:u'}
		Let $X\subseteq\mathbb{R}^d$. For $u \in \mathbb{R}^d$, let $u_{NN}= \text{argmin}_{x \in \overline{X}} \; \| u-x\|_2$. Suppose that $\Pi \in \mathbb{R}^{m\times d}$ provides $\frac{\varepsilon}{6}$-convex hull distortion for $S_X$.  Then, there is $u'\in \mathbb{R}^m$ such that
		\begin{align}
			& \|u'\|\leq \|u-u_{NN}\| \label{eq: constraint_norm},\\
			& \big{|}\langle u',\Pi(x-u_{NN})\rangle - \langle u-u_{NN}, x-u_{NN}\rangle \big{|} \leq \varepsilon \|u-u_{NN}\|\|x-u_{NN}\| \quad \forall \, x \in X.\label{eq: constraint_angles}
		\end{align}
	\end{lemma}
	
	As we noted above, equations (\ref{eq: constraint_norm}) and (\ref{eq: constraint_angles}) are enough to guarantee that a function $f\colon \mathbb{R}^d \longrightarrow \mathbb{R}^{m+1}$ defined as in (\ref{eq: terminal embedding}) satisfies the terminal condition (\ref{eq: terminal condition}). This was proved for finite sets in \cite[Lemma 3.7]{narayanan2019optimal}, and later extended for infinite sets in \cite[Theorem 1.1]{chiclana2024on}.
	
	\begin{lemma}[{\cite[Remark 1.2]{chiclana2024on}}]\label{lem: is terminal}
		Let $X\subseteq\mathbb{R}^d$ and $\varepsilon\in(0,1)$. Let $\Pi\colon \mathbb{R}^d \longrightarrow \mathbb{R}^m$ provide $\frac{\varepsilon}{60}$-convex hull distortion for the units secants $S_X$. For any $u \in \mathbb{R}^d$, let $u_{NN}= \text{argmin}_{x \in \overline{X}} \; \| u-x\|$ and let $u'\in \mathbb{R}^m$ be a point satisfying (\ref{eq: constraint_norm}) and (\ref{eq: constraint_angles}) for $\frac{\varepsilon}{10}$. Define $f\colon \mathbb{R}^d \longrightarrow \mathbb{R}^{m+1}$ by 
		\[f(u)=\begin{cases}
			\big{(}\Pi u_{NN} +u',\sqrt{\|u-u_{NN}\|^2 - \|u'\|^2}\big{)} & \mbox{if }  u \in \mathbb{R}^d \setminus\overline{X};\\
			\big{(}\Pi u,0\big{)} & \mbox{if } u \in \overline{X}.
		\end{cases}\]
		Then, $f$ provides a terminal embedding with distortion $\varepsilon$.
	\end{lemma}
	
	In parallel, Cherapanamjeri and Nelson \cite{Cherapanamjeri2022terminal} proposed a construction of terminal embeddings that can be evaluated in sublinear time. Their approach replaces exact nearest neighbor queries with approximate nearest neighbor structures, leading to a significant improvement in computational efficiency.
	
	While this development is very appealing from an algorithmic perspective, it does not directly address questions of regularity. In particular, their construction still relies on associating to each query point a representative point in $X$, albeit in an implicit and approximate manner. As a result, the embedding may still exhibit discontinuities when this representative point changes. In this work we focus on the nearest neighbor–based construction, as the framework of \cite{Cherapanamjeri2022terminal}, while computationally efficient, does not appear to provide additional regularity properties within the scope of our approach.

	\subsection{Roadmap}
	
	The rest of the paper is structured as follows. 
		In Section \ref{sec 3}, we present the optimization problems that will be used to construct regular terminal embeddings.  Section \ref{sec 4} and Section \ref{sec 5} are then dedicated to proving Theorem \ref{theo: main finite} and Theorem \ref{theo: main infinite}, respectively.

	\section{Constructing regular terminal embeddings via optimization}\label{sec 3}
	
	Since the terminal embeddings presented in Section \ref{sec 2} are essentially constructed point by point, with the selection of a valid $u' \in \mathbb{R}^m$ for each $u \in \mathbb{R}^d$, we potentially compromise regularity (smoothness) properties of the resulting map in the process. However, Lemma \ref{lem: is terminal} shows that any map $u\mapsto u'$ for which $u'$ satisfies equations (\ref{eq: constraint_norm}) and (\ref{eq: constraint_angles}) can be utilized to generate a terminal embedding. In order to construct regular terminal embeddings, we will exploit this freedom and select $u'$ so that the mapping $u\mapsto u'$ is smooth. In this work, we achieve this by selecting $u'$ to be the orthogonal projection of $0$ onto the set of points satisfying (\ref{eq: constraint_norm}) and (\ref{eq: constraint_angles}). Indeed, this allows us to see $u'$ as the solution of a specific optimization problem. Then, we can apply results from optimization theory to analyze the regularity of such a solution. For convenience, we will denote this crucial point $u'$ by $\alpha(u)$ in the finite case, and by $\beta(u)$ in the infinite case. Let us start by introducing the optimization problem in the finite case.
	
	\subsection{The finite case}
	
	Let $X=\{x_1,\ldots,x_n\} \subseteq \mathbb{R}^d$, $\varepsilon \in (0,1)$, and $\Pi \colon \mathbb{R}^d \longrightarrow \mathbb{R}^m$ providing $\frac{\varepsilon}{60}$-convex hull distortion for $S_X$.
	
	We define functions $g_i \colon \mathbb{R}^m \times \mathbb{R}^d \longrightarrow \mathbb{R}$ by
	\begin{align}\label{eq: g_i 1}
		g_i(z,u)&=\langle z,\Pi (x_i-u_{NN})\rangle - \langle u-u_{NN}, x_i-u_{NN}\rangle - \frac{\varepsilon}{10}\|u-u_{NN}\|\|x_i-u_{NN}\|, \, i=1,\ldots,n.\\
		\label{eq: g_i 2}
		g_{n+i}(z,u)&= \langle u-u_{NN},x_i-u_{NN} \rangle - \langle z,\Pi (x_i-u_{NN})\rangle- \frac{\varepsilon}{10}\|u-u_{NN}\|\|x_i-u_{NN}\|, \, i=1,\ldots,n.
	\end{align}

	Observe that functions $g_i$ are well-defined on $\mathbb{R}^m \times \mathbb{R}^d$ since the mapping $u \mapsto u_{NN}$ is defined even for $u \in \mathbb{R}^d$ such that $\arg\min_{x \in \overline{X}} \; \| u-x\|$ contains more than one point (see Remark \ref{rem: u_NN unique}). Fixing $u\in \R^d$,  we call the set of all $z\in R^m$ that satisfy the constraints $g_i(z,u)\leq 0$ for $i=1,\ldots,2n$, the {\em feasible set (for $u$)}, and denote it by $F_u$, i.e.,
		\[
		F_u=\{z\in \R^m \colon g_i(z,u)\leq 0 \quad \forall \, i=1,\ldots,2 n\}.
		\]
	The functions $g_i(\cdot,u)$ are affine for $i=1,\ldots,2n$. Therefore, $F_u$ is a closed and convex (polyhedral) set, which is not empty as Lemma \ref{lem:u'} shows. In particular, the origin has a unique orthogonal projection onto $F_u$, which we denote  $\alpha(u)$. Equivalently, $\alpha(u)$ is the point in $F_u$ with minimal norm, that is, the solution of the following optimization problem:
	
	\begin{equation}\label{optimization problem finite}
		\begin{aligned}
			&\underset{z \in \mathbb{R}^m}{\operatorname{min}} \quad \quad  \|z\| \\
			&\text{s.t.} \quad g_i(z,u)\leq 0,\ i=1,\dots,2n.
		\end{aligned}
		\tag{P$_u$}
	\end{equation}
	Thus, $\alpha \colon \mathbb{R}^d \longrightarrow \mathbb{R}^m$ maps a point $u \in \mathbb{R}^d$ to the solution of the optimization problem (\ref{optimization problem finite}), or, more concretely, 
		\[
		\alpha(u)=\arg\min_{z\in F_u} \|z\|.
		\]
	
	\begin{remark}
		Let $f\colon \mathbb{R}^d \longrightarrow \mathbb{R}^{m+1}$ defined as in (\ref{eq: terminal embedding}), where $u'=\alpha(u)$. Recall that if $\alpha(u)$ satisfies conditions (\ref{eq: constraint_norm}) and (\ref{eq: constraint_angles}), then Lemma \ref{lem: is terminal} guarantees that $f$ is a terminal embedding.
		Clearly, (\ref{eq: constraint_angles}) holds for any point in the feasible set $F_u$ and, in particular, for $\alpha(u)$. Additionally, Lemma~\ref{lem:u'} shows that there is a point in the feasible set $F_u$ satisfying (\ref{eq: constraint_norm}). Since $\alpha(u)$ is the point in $F_u$ with minimal norm, then it also satisfies (\ref{eq: constraint_norm}).
	\end{remark} 
	In order to prove Theorem \ref{theo: main finite}, we will show in Section \ref{sec 4} that $\alpha\colon \mathbb{R}^d \longrightarrow \mathbb{R}^m$ is locally Lipschitz almost everywhere. Then, we will see that this induces H\"{o}lder continuity in the terminal embedding $f$.
	
	\subsection{ The general case}
	
	In this subsection, we extend the previous analysis for arbitrary subsets of $\mathbb{R}^d$. Let $X$ be a subset of $\mathbb{R}^d$, let $S_X$ be its unit secants, defined as in (\ref{eq: S_X}), and $\varepsilon \in (0,1)$. Fix $u \in \mathbb{R}^d$ and recall that $u_{NN}$ (defined in \eqref{eq: u_NN}) is a point from $\overline{X}$ at minimal distance from $u$.
	
	Recall that, by Lemma \ref{lem: is terminal}, for a map $u \mapsto u'$ to generate a terminal embedding $f \colon \mathbb{R}^d \longrightarrow \mathbb{R}^{m+1}$ of the form (\ref{eq: terminal embedding}), it is enough to satisfy (\ref{eq: constraint_norm}) and (\ref{eq: constraint_angles}). When $X$ is an infinite set, these are infinitely many constraints. We can reduce it to a finite number of constraints with the following approach:
	Consider a finite \textit{$\frac{\varepsilon}{4}$-cover} $\mathcal{C}$ of $S_X$, i.e., for any $v \in S_X$ there is $w \in \mathcal{C}$ with $\|v-w\|<\frac{\varepsilon}{4}$.   We claim that it is enough to show that (\ref{eq: constraint_angles}) holds for all $w$ in $\mathcal{C}$.
	
	This is deduced from the proof of \cite[Lemma 3.4]{chiclana2024on}, but we include it here for completeness.
	
	\begin{lemma}\label{lem: u' for cover}
		Let $X\subseteq\mathbb{R}^d$, $\varepsilon \in (0,1)$, and $\mathcal{C}$ an $\frac{\varepsilon}{40}$-cover of $S_X$. Let $\Pi \colon \mathbb{R}^d \longrightarrow \mathbb{R}^m$ provide $\frac{\varepsilon}{240}$-convex hull distortion for $S_X$. Given $u \in \mathbb{R}^d$, assume $u'\in \mathbb{R}^m$ satisfies $\|u'\|\leq \|u-u_{NN}\|$ and
		\begin{equation}\label{eq: cover condition}
			\big{|}  \langle u',\Pi w  \rangle - \langle u-u_{NN}, w \rangle \big{|} \leq \frac{\varepsilon}{30} \|u-u_{NN}\| \quad \forall \, w \in \mathcal{C}.
		\end{equation}
		Then, $\big{|}  \langle u',\Pi(x-u_{NN})  \rangle - \langle u-u_{NN}, x-u_{NN} \rangle \big{|} \leq \frac{\varepsilon}{10} \|u-u_{NN}\|\|x-u_{NN}\|$ for all $x \in X$. 
	\end{lemma} 
	
	\begin{proof}
		Fix $x \in X$ with $x\neq u_{NN}$ and write $w_x= \frac{x-u_{NN}}{\|x-u_{NN}\|} \in S_X$. It now suffices to show that
		\[ 
		\big{|}  \langle u',\Pi w_x  \rangle - \langle u-u_{NN}, w_x \rangle \big{|} \leq \frac{\varepsilon}{10} \|u-u_{NN}\|.
		\]
		Take $w \in \mathcal{C}$ with $\|w_x - w\|\leq \frac{\varepsilon}{40}$. A straightforward application of the triangle inequality shows that $\left |  \langle u',\Pi w_x  \rangle - \langle u-u_{NN}, w_x \rangle \right |$ is less than or equal to
		\[ \big{|}  \langle u',\Pi w_x  \rangle - \langle u', \Pi w \rangle \big{|} + \big{|}  \langle u',\Pi w  \rangle - \langle u-u_{NN}, w \rangle \big{|} + \big{|}  \langle \langle u-u_{NN}, w  \rangle - \langle u-u_{NN}, w_x \rangle \big{|}.\]
		By assumption \eqref{eq: cover condition}, the middle term is bounded above by $\frac{\varepsilon}{30} \|u-u_{NN}\|$. Similarly, the Cauchy-Schwarz inequality (and our choice of $w_x$) shows that the last term is  bounded above  by $\frac{\varepsilon}{40}\|u-u_{NN}\|$. Thus, it remains to bound the first term. Observe that
		\begin{align*} \big{|}  \langle u',\Pi w_x  \rangle - \langle u', \Pi w \rangle \big{|} &\leq \|\Pi(w_x-w)\| \|u'\| \leq 2\left \|\frac{\Pi(w_x-w)}{2}\right \| \|u-u_{NN}\| \\
			& \leq 2\left (\left \|\frac{w_x-w}{2}\right \| + \frac{\varepsilon}{240}\right ) \|u-u_{NN}\| \leq \frac{\varepsilon}{30}\|u-u_{NN}\|,
		\end{align*}
		where the third inequality follows from the fact that $\Pi$ provides an $\frac{\varepsilon}{240}$-convex-hull distortion of $S_X$ and that $\frac{w_x-w}{2}\in {\rm conv}(S_X)$; the latter being true as $-w_x\in S_X$.
	\end{proof}
	
	Let $X$ be a subset of $\mathbb{R}^d$, $\varepsilon \in (0,1)$, and $\mathcal{C}=\{w_1,\ldots,w_\ell\}$ be an $\frac{\varepsilon}{40}$-cover of $S_X$. Let $\Pi \colon \mathbb{R}^d \longrightarrow \mathbb{R}^m$ provide $\frac{\varepsilon}{240}$-convex hull distortion for $S_X$. Define $\tilde{g}_i \colon \R^m \times \mathbb{R}^d \longrightarrow \R$ by
	\begin{align}\label{eq: g_i tilde 1}
		\tilde{g}_i(z,u)&=\langle z,\Pi w_i \rangle - \langle u-u_{NN}, w_i\rangle - \frac{\varepsilon}{30}\|u-u_{NN}\|, i=1,\ldots,\ell.\\
		\label{eq: g_i tilde 2}
		\tilde{g}_{\ell+i}(z,u)&= \langle u-u_{NN}, w_i\rangle - \langle z,\Pi w_i \rangle- \frac{\varepsilon}{30}\|u-u_{NN}\|, i=1,\ldots,\ell.
	\end{align}
	We write $\widetilde{F}_u$ for the feasible set of a point $u \in \mathbb{R}^d$ for the constraints $\tilde{g}_i(z,u)\leq 0$, that is,
	\[\widetilde{F}_u=\{z\in \R^m \colon \tilde{g}_i(z,u)\leq 0 \quad \forall \, i=1,\ldots,2\ell\}.\]
	As in the previous subsection, observe that $\widetilde{F}_u$ is a closed and convex set, which is not empty as Lemma \ref{lem:u'} shows. Hence, the origin has a unique orthogonal projection onto $\widetilde{F}_u$, which we denote $\beta(u)$. Equivalently, $\beta(u)$ is the point in $\widetilde{F}_u$ with minimal norm, that is, the solution of the following optimization problem:
	
	\begin{equation}\label{optimization problem infinite}
		\begin{aligned}
			&\underset{z \in \mathbb{R}^m}{\operatorname{minimize}} \quad \quad  \|z\| \\
			&\text{subject to } \quad \tilde{g}_i(z,u)\leq 0,\ i=1,\dots,2 \ell.
		\end{aligned}
		\tag{$\widetilde{\operatorname{P}}_u$}
	\end{equation}
	Thus, $\beta \colon \mathbb{R}^d \longrightarrow \mathbb{R}^m$ maps a point $u \in \mathbb{R}^d$ to the solution of the optimization problem (\ref{optimization problem infinite}). 
	
	\begin{remark}
		Let $f\colon \mathbb{R}^d \longrightarrow \mathbb{R}^{m+1}$ defined as in (\ref{eq: terminal embedding}), where $u'=\beta(u)$. Recall that if $\beta(u)$ satisfies conditions (\ref{eq: constraint_norm}) and (\ref{eq: cover condition}), then Lemma \ref{lem: u' for cover} and Lemma \ref{lem: is terminal} guarantee that $f$ is a terminal embedding.
		Clearly, (\ref{eq: cover condition}) holds for any point in the feasible set $\widetilde{F}_u$ and, in particular, for $\beta(u)$. Additionally, Lemma~\ref{lem:u'} shows that there is a point in the feasible set $\widetilde{F}_u$ satisfying (\ref{eq: constraint_norm}). Since $\beta(u)$ is the point in $\widetilde{F}_u$ with minimal norm, then it also satisfies (\ref{eq: constraint_norm}).
	\end{remark}
	When $X$ has positive reach $\tau_X>0$, and we restrict to $\widetilde{X}=X+B(0,\tau_X/2)$, we will show in Section~\ref{sec 5} that $\beta \colon \widetilde{X} \longrightarrow \mathbb{R}^m$ is locally $\frac{1}{2}$-H\"{o}lder, which induces $\frac{1}{4}$-H\"{o}lder continuity in the terminal embedding defined as in (\ref{eq: terminal embedding}), where $u'=\beta(u)$.

	\section{Finite case (proof of Theorem \ref{theo: main finite})}\label{sec 4}
	
	This section is dedicated to proving Theorem \ref{theo: main finite}. Recall that $\alpha\colon \mathbb{R}^d \longrightarrow \mathbb{R}^m$ maps a point $u \in \mathbb{R}^d$ to the solution of the optimization problem (\ref{optimization problem finite}). Our approach utilizes results from optimization theory to analyze the regularity of the map $\alpha$. Then, we will infer regularity for a terminal embedding of the form (\ref{eq: terminal embedding}), where $u'=\alpha(u)$. We start by introducing some notation and preliminary results.
	
	\subsection{Constraint qualifications}
	
	Let $U \subseteq \mathbb{R}^d$ be open, and consider $g_i\colon \mathbb{R}^m \times \mathbb{R}^d \longrightarrow \mathbb{R}$ so that the functions $g_i(\cdot,u)$ are continuous and convex for all $u \in \mathbb{R}^d$, and continuously differentiable for $u \in U$, for $i=1,\ldots, \ell$. For $u \in \mathbb{R}^d$, define the feasible set
		\[
		\mathcal F_u=\{z\in \R^m \colon g_i(z,u)\leq 0 \quad \forall \, i=1,\ldots,\ell\},
		\]
		which is closed and convex, and consider the following optimization problem: 
		\begin{equation}\label{eq:CvxP}
			\min \|z\| \st z\in \mathcal F_u.
		\end{equation}
		Clearly, the (unique) solution of \eqref{eq:CvxP} is the projection of $0$ onto $\mathcal F_u$, and we may thus tap into the rich theory developed for {\em projections onto moving sets}; we mainly draw from results collected in the monograph by Facchinei and Pang \cite{facchinei2003finite}. 
		
		To infer smoothness properties of the projection onto a parameterized (closed, convex) set, one needs to ensure   \emph{constraint qualifications (CQ)} for the feasible set at the point in  question. 
		To this end, define the  \textit{active set} of a point $z \in \mathcal F_u$  by 	$I_u(z) = \{i\in \{1,\ldots,\ell\}\colon g_i(z,u)=0\}.$
		
		\begin{itemize}
			\item We say that $\operatorname{MFCQ}$ \textit{(Mangasarian Fromovitz CQ)} holds at $z \in F_u$ for \eqref{eq:CvxP} if there is $w\in\R^m$ such that 
			\[\nabla_{1} g_i(z,u)^T w <0 \quad \forall \, i \in I_u(z),\]
			where $\nabla_{1}$ denotes the gradient with respect to the first variable.
			\item We say that $\operatorname{SCQ}$ \textit{(Slater CQ)} holds for \eqref{eq:CvxP} if there is $\hat{z} \in \R^m$ so that
			\[g_i(\hat{z},u)<0 \quad \forall \, i =1,\ldots,\ell.\]
			\item We say that $\operatorname{CRCQ}$ \textit{(Constant rank CQ)} holds at $z \in F_u$ for \eqref{eq:CvxP} if there is $\varepsilon>0$ such that $\forall \, \tilde{I} \subseteq I_u(z)$ and $\forall \, y \in F_u\cap B(z,\varepsilon)$ we have that $\{\nabla_1 g_i(y,u) \colon i \in \tilde{I}\}$ has constant rank.
		\end{itemize}
		The following result will be crucial for us. Recall that a function $\alpha \colon \mathbb{R}^d \longrightarrow \mathbb{R}^m$ is \textit{locally Lipschitz} at $ u \in \mathbb{R}^d$ if there is a neighborhood $U$ of $u$ and a constant $c>0$ so that
		\[ \|\alpha(v)-\alpha(w)\| \leq c \|v-w\| \quad \forall \, v,w \in  U.\]

		\begin{theorem}\label{theo: tim}
			Let $\alpha\colon \mathbb{R}^d \longrightarrow \mathcal \mathbb{R}^m$ be the solution map of \eqref{eq:CvxP}, i.e.,
			$
			\alpha(u)=\arg\min_{u\in \mathcal F_u}\|z\|,
			$
			and assume that $g_i$ is twice continuously differentiable in a neighborhood of $(\alpha(\bar u),\bar u)$. If $\operatorname{MFCQ}$ and $\operatorname{CRCQ}$ hold at $\alpha(\bar u)$ for \eqref{eq:CvxP}, then $\alpha$ is locally Lipschitz continuous at $\bar u$.
		\end{theorem}
		
		\begin{proof} The proof follows immediately from \cite[Theorem 4.7.5]{facchinei2003finite} realizing that $\alpha$ is (on $U$) the (Euclidean) projector of $0$ onto $\mathcal F_u$.
		\end{proof}
	
		\begin{remark}
			Under the assumptions of Theorem \ref{theo: tim}, \cite[Theorem 4.7.5]{facchinei2003finite} actually shows that $\alpha$ is piecewise continuously differentiable near $\bar u$ in the following sense: $\alpha$ is continuous, and there exist an open neighborhood $U$ of $\bar u$, and a finite family of $\operatorname{C}^1$ functions $\{\alpha^1,\alpha^2,\ldots,\alpha^k\}$ defined on $U$ such that $\alpha(v)$ is an element of $\{\alpha^1(v),\ldots,\alpha^k(v)\}$ for all $v \in U$. This statement is stronger since $\operatorname{C}^1$ functions are locally Lipschitz.
		\end{remark}

		In order to apply this result to \eqref{optimization problem finite}, we need to verify it satisfies the necessary constraint qualifications at the point of question. To this end, the following standard result is useful (see \cite[Proposition 3.2.7]{facchinei2003finite}).
		
		\begin{lemma}\label{lem: SCQ implies MFCQ} If SCQ holds for \eqref{eq:CvxP}, then MFCQ holds at every point of the feasible set $\mathcal F_u$ of \eqref{eq:CvxP}.
		\end{lemma}
		
	
	\subsection{Constraint qualifications hold for the optimization problem (\ref{optimization problem finite})}
	
	Let $X=\{x_1,\ldots,x_n\}\subseteq \mathbb{R}^d$. For $k=1,\ldots,n$, let $V_k$ be the \textit{Voronoi cell} of $x_k$, i.e., $V_k=\{u \in \mathbb{R}^d \colon \|u-x_k\|<\|u-x_j\| \mbox{ for all } j\neq k\}$. Fix $u \in V_k$ with $k=1,\ldots,n$, and recall that $\alpha\colon \mathbb{R}^d \longrightarrow \mathbb{R}^m$ is the solution map of (\ref{optimization problem finite}).
	
	First, we notice that the functions $g_i$ given by (\ref{eq: g_i 1}) and (\ref{eq: g_i 2}) meet the regularity requirements of Theorem \ref{theo: tim}. Indeed, if $U$ is an open neighborhood of $u$ contained in $V_k$, then $v_{NN}=x_k$ for any $v \in U$. Therefore, the $g_i$'s are twice continuously differentiable in $\mathbb{R}^m \times V_k$. Moreover, the $g_i$'s are affine in the first variable, and thus, are continuous and convex. Furthermore, from the affine nature of $g_i(\cdot,u)$, it is easy to deduce that $\operatorname{CRCQ}$ holds at $\alpha(u)$ for (\ref{optimization problem finite}) (see the discussion before \cite[Lemma 3.2.8]{facchinei2003finite}). It remains to show that $\operatorname{MFCQ}$ holds at $\alpha(u)$ for (\ref{optimization problem finite}). In view of Lemma \ref{lem: SCQ implies MFCQ}, it suffices to show that $\operatorname{SCQ}$ holds for the optimization problem (\ref{optimization problem finite}). Notice that for $u \in V_k$, i.e. $u_{NN}=x_k$, the constraints $g_k(\cdot,u)$ and $g_{n+k}(\cdot,u)$ given by (\ref{eq: g_i 1}) and (\ref{eq: g_i 2}) vanish. Therefore, we cannot expect $\operatorname{SCQ}$ to hold for (\ref{optimization problem finite}). This issue can be easily avoided by removing these two trivial constraints, which does not alter the feasible set locally around the point of question.

	\begin{lemma}\label{lem: MFCQ holds at u'}
		Fix $u \in V_k\setminus \{x_k\}$ for some $k=1,\ldots,n$. Consider the optimization problem obtained after removing the constraints $g_k$ and $g_{n+k}$ from (\ref{optimization problem finite}), and denote its solution by $\alpha_k(u)$. Then, $\operatorname{SCQ}$ holds for this optimization problem. Consequently, $\operatorname{MFCQ}$ holds at $\alpha_k(u)$ for this optimization problem.
	\end{lemma}
	
	\begin{proof}
		First, observe that $\|u-u_{NN}\|\|x_i-u_{NN}\| > 0$ for any $i=1,\ldots,n$ with $i\neq k$. Therefore, Lemma \ref{lem:u'} guarantees the existence of a point $\hat{z} \in \R^m$ satisfying
		\begin{align*}\big{|}\langle \hat{z},\Pi (x_i-u_{NN})\rangle - \langle u-u_{NN},x_i-u_{NN} \rangle \big{|} &\leq \frac{\varepsilon}{20}\|u-u_{NN}\|\|x_i-u_{NN}\|\\
			& < \frac{\varepsilon}{10}\|u-u_{NN}\|\|x_i-u_{NN}\|
		\end{align*}
		for all $i=1,\ldots,n$ with $i\neq k$. This implies that $g_i(\hat{z},u)<0$ for all $i=1,\ldots,2n$ with $i\neq k,n+k$, whence $\operatorname{SCQ}$ holds for the optimization problem. The rest of the statement follows from Lemma \ref{lem: SCQ implies MFCQ}.
	\end{proof}
	

	\subsection{Proof of Theorem \ref{theo: main finite}}
	
	First, we show that the solution map $\alpha\colon \mathbb{R}^d \longrightarrow \mathbb{R}^m$ of the optimization problem (\ref{optimization problem finite}) is locally Lipschitz almost everywhere. This follows as an application of Theorem \ref{theo: tim}.
	
	\begin{corollary}\label{cor: u' is local lips}
		Let $X=\{x_1,\ldots,x_n\}$ be a finite subset of $\mathbb{R}^d$, let $V_k$ be the Voronoi cell of $x_k$ for $k=1,\ldots,n$, and write $V=\cup_{i=1}^n V_i$. Let $\alpha \colon \mathbb{R}^d  \longrightarrow \mathbb{R}^m$ be the solution map of \eqref{optimization problem finite}. 
		Then, $\alpha$ is locally Lipschitz  at every point in $V \setminus X$. In particular, $\alpha$ is locally Lipschitz almost everywhere on $\mathbb{R}^d$.
	\end{corollary}
	
	\begin{proof}
		Fix $k=1,\ldots,n$ and $u=\bar u \in V_k\setminus\{x_k\}$. Observe that $v_{NN}=x_k$ for any $v \in V_k$. Thus, the functions $g_i$ given by (\ref{eq: g_i 1}) and (\ref{eq: g_i 2}) are twice continuously differentiable on $\mathbb{R}^m \times V_k$ and affine (hence continuous and convex) in the first variable. Consider the optimization problem obtained from removing the constraints $g_k$ and $g_{n+k}$ from (\ref{optimization problem finite}), and let $\alpha_k\colon \mathbb{R}^d \longrightarrow \mathbb{R}^m$ denote its solution map. Then Lemma \ref{lem: MFCQ holds at u'} shows that MFCQ holds at $\alpha_k(u)$ for this optimization problem. Additionally, CRCQ also holds at $\alpha_k(u)$ for it since $g_i(\cdot,u)$ are affine functions for $i=1,\ldots,2n$ with $i\neq k,n+k$. Therefore, we can apply Theorem \ref{theo: tim} to obtain that $\alpha_k$ is locally Lipschitz at $u$.
		
		Recall that $\alpha \colon \mathbb{R}^d \longrightarrow \mathbb{R}^m$ is the solution map of (\ref{optimization problem finite}). Observe that for any $v \in V_k$, the constraints $g_k(\cdot,v)$ and $g_{n+k}(\cdot,v)$ vanish since $v_{NN}=x_k$. Therefore, removing these constraints from the optimization problem does not change the feasible set. In particular, $\alpha(v)=\alpha_k(v)$ for any $v \in V_k$. In other words, the mappings $\alpha$ and $\alpha_k$ are equal on $V_k$.
		
		Finally, we show that $\mathbb{R}^d \setminus V=:\stcomp{V}$ has measure zero. Observe that if $u \in \stcomp{V}$, then there are $i$, $j \in \{1,\ldots,n\}$ with $i\neq j$ such that $\|u-x_i\|=\|u-x_j\|$. Therefore, we have $\stcomp{V} \subseteq \cup_{i<j=1}^n V_{i,j}$, where 
    \[V_{i,j} = \{ u \in \mathbb{R}^d \colon \|u-x_i\|=\|u-x_j\|\} \quad \forall \, i<j=1,\ldots,n.\]

    Let $H$ be the hyperplane formed by all $u \in \mathbb{R}^d$ that are orthogonal to $w\coloneq (x_j-x_i)/2$. We claim that $V_{i,j}$ is contained in $(x_i+x_j)/2 + H$, which implies that $V_{i,j}$ has measure zero, from where it follows that $\stcomp{V}$ also has measure zero. To prove the claim,  it suffices to show that if $u +(x_i+x_j)/2 \in V_{i,j}$ then $u \in H$. Notice that $u +(x_i+x_j)/2 - x_i = u+w$ and $u +(x_i+x_j)/2 - x_j = u-w$. Therefore, under such assumption we have
    \[ \|u-w\|^2=\|u+w\|^2.\]
    Thus, the polarization identity implies that
    \[ \langle u,w\rangle = \langle u,-w\rangle=-\langle u,w\rangle,\]
    and we conclude that $u$ and $w$ are orthogonal.
	\end{proof}

	Notice that the approach used to prove Corollary \ref{cor: u' is local lips} does not work for points in $\stcomp{V}$. Indeed, the mapping $u \mapsto u_{NN}$ is not continuous at such points. Therefore, the constraints $g_i$ do not satisfy the regularity requirements to apply Theorem \ref{theo: tim}. Similarly, the mapping $u \mapsto \|u-u_{NN}\|$ is not differentiable at points $u \in \mathbb{R}^d$ with $u=u_{NN}$. Thus, our approach also fails for points in $X$.

	We are finally ready to prove Theorem \ref{theo: main finite}. 
	
	\begin{proof}[Proof of Theorem \ref{theo: main finite}]
		Let $\varepsilon \in (0,1)$. Recall that the terminal embedding $f\colon \mathbb{R}^d \longrightarrow \mathbb{R}^{m+1}$ provided by Lemma \ref{lem: is terminal} is defined as
		\[f(u)=\begin{cases}
			\big{(}\Pi u_{NN} +\alpha(u),\sqrt{\|u-u_{NN}\|^2 - \|\alpha(u)\|^2} \big{)} & \mbox{if }  u \in \mathbb{R}^d \setminus \overline{X};\\
			\big{(}\Pi u ,0\big{)} & \mbox{if } u \in \overline{X},
		\end{cases}\]
		where $\Pi \in \mathbb{R}^{d \times m}$ provides $\frac{\varepsilon}{60}$-convex hull distortion for $S_X$, $u_{NN}$ is a point from $\overline{X}$ at minimal distance from $u$, and $\alpha \colon \mathbb{R}^d \longrightarrow \mathbb{R}^m$ assigns $u \in \mathbb{R}^d$ to the solution of the optimization problem (\ref{optimization problem finite}). Let $u \in V \setminus X$ and write $x_k=u_{NN}$ and $V_k$ for its Voronoi cell. Let $p\colon \mathbb{R}^d \longrightarrow \mathbb{R}$ defined by $p(u)=\|u-u_{NN}\|^2 - \|\alpha(u)\|^2$. Then, for any points $v$, $w$ in $V_k$ we have that
		\[\|f(v)-f(w)\|^2= \|\alpha(v)-\alpha(w)\|^2 + \left |\sqrt{p(v)}-\sqrt{p(w)}\right |^2 \leq \|\alpha(v)-\alpha(w)\|^2 + |p(v)-p(w)|. \]
		As Corollary \ref{cor: u' is local lips} shows, $\alpha$ is locally Lipschitz at $u$, that is, there are a constant $\widetilde{D}_u>0$ and $r_u \in (0, 1/2)$ such that $\|\alpha(v)-\alpha(w)\| \leq \widetilde{D}_u \|v-w\|$ for any $v$, $w \in B(u,r_u)\cap V_k$. Thus, it remains to bound $|p(v)-p(w)|$. First, note that
		\begin{align*}
			\big{|}\|v-v_{NN}\|^2 -\|w-w_{NN}\|^2\big{|} &= \big{(}\|v-x_k\| +\|w-x_k\|\big{)}\big{|}\|v-x_k\| -\|w-x_k\|\big{|}\\
			&\leq \big{(}\|v-x_k\| +\|w-x_k\|\big{)}\|v-w\| \leq \big{(}1+2\|u-x_k\|\big{)}\|v-w\|.
		\end{align*}
		Furthermore, we have
		\begin{align*}\big{|}\|\alpha(v)\|^2 - \|\alpha(w)\|^2\big{|}&=\big{(}\|\alpha(v)\|+\|\alpha(w)\|\big{)}\big{|}\|\alpha(v)\|-\|\alpha(w)\|\big{|} \leq \big{(}\|\alpha(v)\|+\|\alpha(w)\|\big{)}\|\alpha(v)-\alpha(w)\|\\
			&\leq \big{(}\|v-x_k\|+\|w-x_k\|\big{)} \|\alpha(v)-\alpha(w)\|\leq \big{(}1+2\|u-x_k\|\big{)}\widetilde{D}_u\|v-w\|.
		\end{align*}
		Therefore, the triangle inequality yields
		\[ |p(v)-p(w)| \leq (1+\widetilde{D}_u)(1+2\|u-x_k\|)\|v-w\|.\]
		We conclude that for any $v$, $w \in B(u,r_u)$ we have
		\[ \|f(v)-f(w)\| \leq \|\alpha(v)-\alpha(w)\| + |p(v)-p(w)|^{1/2} \leq \left (\widetilde{D}_u + \sqrt{(1+\widetilde{D}_u)(1+2\|u-x_k\|)}\right )\|v-w\|^{1/2}. \qedhere\]
	\end{proof}

	\section{General case (proof of Theorem \ref{theo: main infinite})}\label{sec 5}
	
	This section is dedicated to proving Theorem \ref{theo: main infinite}. Recall that $\beta \colon \mathbb{R}^d \longrightarrow \mathbb{R}^m$ maps $u \in \mathbb{R}^d$ to the solution of the optimization problem  (\ref{optimization problem infinite}). As in Section \ref{sec 4}, we first analyze the regularity of the map $\beta$. Then, we infer regularity for a terminal embedding of the form (\ref{eq: terminal embedding}) from it, where $u'=\beta(u)$. It is important to note that for general sets $X \subseteq \mathbb{R}^d$, the functions $\tilde{g_i}$ appearing in the optimization problem (\ref{optimization problem infinite}) might not be differentiable. Indeed, the map $u \mapsto u_{NN}$ is generally not differentiable. Therefore, we cannot apply Theorem \ref{theo: tim} to obtain that $\beta$ is locally Lipschitz. However, when we restrict to $\widetilde{X}=X+B(0,\tau_X/2)$, part (8) of \cite[Theorem 4.8]{federer1959curvature}) shows that $u\mapsto u_{NN}$ is Lipschitz with
	\begin{equation}\label{eq: u_NN Lipschitz} \|u_{NN}-v_{NN}\|\leq 2 \|u-v\| \quad \forall \, u,v \in \widetilde{X}.
	\end{equation}
	We will see that (\ref{eq: u_NN Lipschitz}) is enough to guarantee that $\beta$ is $\frac{1}{2}$-H\"{o}lder continuous on $\widetilde{X}$.
	
	\begin{theorem}\label{theo: u' holder}
		Let $X$ be a subset of $\mathbb{R}^d$ with positive reach $\tau_X$ and write $\widetilde{X}=X+ B(0,\tau_X/2)$. For $\varepsilon \in (0,1)$, let $\beta \colon \widetilde{X} \longrightarrow \mathbb{R}^m$ assign $u \in \widetilde{X}$ to the solution of the optimization problem (\ref{optimization problem infinite}). Then,  $\beta$ is locally $\frac{1}{2}$-H\"{o}lder continuous on $\widetilde{X}$, that is, for any $u \in \widetilde{X}$ there is a neighborhood $U' \subseteq \mathbb{R}^d$ of $u$ so that for all $w, v \in U'$ we have
		\begin{equation}\label{eq: theo holder}
			\|\beta(w)-\beta(v)\| \leq C'_u \|\beta(w)-\beta(v)\|^{1/2},
		\end{equation}
		where $C'_u>0$ is a constant that only depends on the distance between $u$ and the set $X$ and on $\varepsilon$.
	\end{theorem}
	
	Theorem \ref{theo: main infinite} will follow from Theorem \ref{theo: u' holder}, since (\ref{eq: theo holder}) implies that a terminal embedding of the form (\ref{eq: terminal embedding}), where $u'=\beta(u)$, is $\frac{1}{4}$-H\"{o}lder continuous (see Subsection~\ref{ss:proof_main_infinite}).
	
	\subsection{Theorem \ref{theo: u' holder} when $w=u$.}
	Before proving Theorem \ref{theo: u' holder}, we will prove the following special case.
	
	\begin{lemma}\label{lem: holder}
		Let $X$ be a subset of $\mathbb{R}^d$ with positive reach $\tau_X$ and write $\widetilde{X}=X+ B(0,\tau_X/2)$. For $\varepsilon \in (0,1)$, let $\beta \colon \widetilde{X} \longrightarrow \mathbb{R}^m$ assign $u \in \widetilde{X}$ to the solution of the optimization problem (\ref{optimization problem infinite}). Then, there is a neighborhood $\widetilde{U} \subseteq \mathbb{R}^d$ of $u$ so that for all $v \in \widetilde{U}$ we have
		\begin{equation}\label{eq: theo holder w=u}
			\|\beta(u)-\beta(v)\| \leq \widetilde{C}_u \|u-v\|^{1/2},
		\end{equation}
		where $\widetilde{C}_u>0$ is a constant that only depends on the distance between $u$ and the set $X$ and on $\varepsilon$.
	\end{lemma}
	
	\begin{remark}\label{rem: constants for lemma}
		We are able to find explicit expressions for the neighborhood $\widetilde{U}$ and the constant $\widetilde{C}_u$ in Lemma \ref{lem: holder}. If $u \in \overline{X}$, then (\ref{eq: trivial case}) will show that we can take $\widetilde{U}=B(u,1)\cap \widetilde{X}$, and $\widetilde{C}_u=3$. If $u \notin \overline{X}$, the proof of Lemma \ref{lem: holder} will show that we can take
		\[ r_u =\min\left\{1,\frac{\varepsilon \|u-u_{NN}\|}{480}\right\}, \quad \quad  \widetilde{U}=B(u,r_u) \cap \widetilde{X}, \quad \quad \widetilde{C}_u= \frac{1}{\varepsilon}\max\left \{1200, 480+128\|u-u_{NN}\|^{1/2}\right \}.\]
	\end{remark}
	
	The proof of Lemma \ref{lem: holder} will be broken into several parts. Our first lemma shows that the functions $\tilde{g_i}$ appearing in the optimization problem (\ref{optimization problem infinite}) are Lipschitz in the second variable when restricted to $\widetilde{X}$. Moreover, their Lipschitz constant does not depend on the first variable $z \in \mathbb{R}^m$. Recall that the \textit{Lipschitz constant} of a function $f \colon \widetilde{X} \longrightarrow \mathbb{R}$ is
	\[ \|f\|_L= \sup \left \{ \frac{|f(x)-f(y)|}{\|x-y\|} \colon x \neq y \in \widetilde{X} \right \}.\]
	
	\begin{lemma}\label{lem: Delta}
		Let $z \in \mathbb{R}^m$ and define $h_i \colon \widetilde{X} \longrightarrow \mathbb{R}$ by $h_i(u)=\tilde{g}_i(z,u)$ for $i=1,\ldots,2\ell$, where $\tilde{g}_i$ are defined as in (\ref{eq: g_i tilde 1}) and (\ref{eq: g_i tilde 2}). Then, $h_i$ are Lipschitz and $\|h_i\|_L \leq 4$, that is,
		\[ |h_i(v)-h_i(w)|\leq 4\|v-w\| \quad \forall \, v,w \in \widetilde{X}.\]
	\end{lemma}
	
	\begin{proof}
		Recall that $\tilde{g}_i(z,u)$  is defined by
		\[ \tilde{g}_i(z,u)=\langle z,\Pi w_i \rangle - \langle u-u_{NN}, w_i\rangle - \frac{\varepsilon}{30}\|u-u_{NN}\|\]
		for $i=1,\ldots,\ell$. Therefore, for any $z \in \mathbb{R}^m$ and $v$, $w \in \widetilde{X}$ we have
		\begin{align*}
			|\tilde{g}_i(z,v) - \tilde{g}_i(z,w)|&= \left | \langle v-v_{NN}, w_i\rangle + \frac{\varepsilon}{30}\|v-v_{NN}\| - \langle w-w_{NN}, w_i\rangle - \frac{\varepsilon}{30}\|w-w_{NN}\| \right |\\
			&\leq \big{|} \langle (v-v_{NN})-(w-w_{NN}), w_i\rangle \big{|} +  \frac{\varepsilon}{30} \big{|} \|v-v_{NN}\| - \|w-w_{NN}\| \big{|}\\
			&\leq \left (1+\frac{\varepsilon}{30}\right ) \|(v-v_{NN})-(w-w_{NN})\|.
		\end{align*}
		Notice that the mapping $v \mapsto v-v_{NN}$ is Lipschitz. In fact, (\ref{eq: u_NN Lipschitz}) gives
		\begin{equation}\label{eq: u-u_NN Lipschitz} 
			\|(v-v_{NN})-(w-w_{NN})\| \leq \|v-w\| + \|v_{NN}-w_{NN}\| \leq 3 \|v-w\| \quad \forall \, v,w \in \widetilde{X}.
		\end{equation}
		Therefore, we conclude that
		\[ |\tilde{g}_i(z,v) - \tilde{g}_i(z,w)| \leq 3\left (1+\frac{\varepsilon}{30}\right )\|v-w\| \leq 4\|v-w\|. \]
		The case $i=\ell+1,\ldots,2\ell$ is analogous.
	\end{proof}

	Next, we show that there is a point $z^*$ for which all the constraints $\tilde{g}_i(z^*,u)\leq 0$ of the optimization problem (\ref{optimization problem infinite}) are strongly satisfied for all $i=1,\ldots,2\ell$. Intuitively, we can think of $z^*$ as a point in the interior of the feasible set $\widetilde{F}_u$.
	
	\begin{lemma}\label{lem: z*}
		There is $z^* \in \mathbb{R}^m$ with $\|z^*\|\leq \|u-u_{NN}\|$ so that 
		\begin{equation}\label{eq: z*}
			\tilde{g}_i(z^*,u)\leq -\frac{\varepsilon}{60} \|u-u_{NN}\| \quad \forall \, i=1,\ldots,2\ell.
		\end{equation}
	\end{lemma}
	
	\begin{proof}
		Since the embedding $\Pi$ provides $\frac{\varepsilon}{60}$-convex hull distortion for $S_X$, we can apply Lemma \ref{lem: narayan u'} to the points $\{w_1,\ldots,w_\ell\} \subseteq S_X$ and $u-u_{NN} \in \mathbb{R}^d$ to obtain that there is a point $z^* \in \mathbb{R}^m$ such that $\|z^*\|\leq \|u-u_{NN}\|$ and
		\[ |\langle z^*,\Pi w_i \rangle - \langle u-u_{NN}, w_i\rangle | \leq \frac{\varepsilon}{60} \|u-u_{NN}\|.\]
		The result follows from substituting the above inequality in the equation of $\tilde{g}_i(z,u)$.
	\end{proof}
	
	Fix $u \in \mathbb{R}^d$ and recall that $\beta \colon \widetilde{X} \longrightarrow \mathbb{R}^m$ assigns $u \in \mathbb{R}^d$ to the solution of the optimization problem (\ref{optimization problem infinite}). Without loss of generality we may assume that $\|u-u_{NN}\|>0$. Indeed, since $\beta(u)$ verifies (\ref{eq: constraint_norm}), where $u'=\beta(u)$, then $u=u_{NN}$ implies $\beta(u)=0$. Hence, for any $v \in  \widetilde{X}$ with $\|u-v\|\leq 1$ we have
	\begin{equation}\label{eq: trivial case} \|\beta(u)-\beta(v)\|=\|\beta(v)\|\leq \|v-v_{NN}\| =\|(v-v_{NN}) - (u-u_{NN})\|\leq 3\|u-v\| \leq 3\|u-v\|^{1/2},
	\end{equation}
	where we have used $\|u-u_{NN}\|=0$ and (\ref{eq: u-u_NN Lipschitz}). 
	
	Now, define
	\begin{equation}\label{eq: r} r_u=\min\left\{1,\frac{\varepsilon \|u-u_{NN}\|}{480}\right\} \quad \mbox{and} \quad  \widetilde{U}=B(u,r_u) \cap \widetilde{X}.
	\end{equation}
	Our goal is to show that for any $v \in \widetilde{U}$ we have $\|\beta(u)-\beta(v)\|\leq \widetilde{C}_u \|u-v\|^{1/2}$ for some constant $\widetilde{C}_u>0$. Fix $v \in \widetilde{U}$ and define
	\begin{equation}\label{eq: delta}
		\delta_v = \frac{240\|u-v\|}{\varepsilon\|u-u_{NN}\|} \leq \frac{1}{2} \quad \mbox{and} \quad z_v = (1-\delta_v)\beta(u)+\delta_v z^*,
	\end{equation}
	where $z^*$ is the point provided by Lemma \ref{lem: z*}, whence it satisfies (\ref{eq: z*}). Intuitively, $z_v$ is a point close to $\beta(u)$ that belongs to the interior of the feasible set $F_u$. The strategy to prove Lemma \ref{lem: holder} will be to take $v$ close enough to $u$ so that $z_v$ also belongs to the feasible set $F_v$. Then, we can use the Lipschitz condition of the constraints to guarantee that $z_v$ is also close to $\beta(v)$. Finally, we will use the triangle inequality to estimate $\|\beta(u)-\beta(v)\|$ as
	\begin{equation}\label{eq: triangle inequality } 
		\|\beta(u)-\beta(v)\|\leq \|\beta(u)-z_v\| + \|z_v-\beta(v)\|.
	\end{equation}
	The estimation for $\|\beta(u)-z_v\|$ easily follows from Lemma~\ref{lem: z*} since for any $v \in \widetilde{U}$ we have
	\begin{equation}\label{eq: |u'-zv|} 
		\|\beta(u)-z_v\| = \delta_v \|\beta(u) - z^*\| \leq \delta_v (\|\beta(u)\|+\|z^*\|)\leq 2\delta_v\|u-u_{NN}\|=\frac{480}{\varepsilon}\|u-v\|.
	\end{equation}
	We will estimate $\|z_v-\beta(v)\|$ through as series of lemmata. First, we show that $z_v$ belongs to the feasible set $F_v$.
	
	\begin{lemma}\label{lem: z in F_v}
		Let $v \in \widetilde{U}$. Then,  $\tilde{g}_i(z_v,v)\leq 0$ for $i=1,\ldots,2\ell$.
	\end{lemma}
	
	\begin{proof}
		Recall that the functions $\tilde{g}_i(\cdot,v)$ are affine for $i=1,\ldots,2\ell$. Therefore, 
		\[ \tilde{g}_i(z_v,v)= (1-\delta_v) \tilde{g}_i(\beta(u),v) + \delta_v \tilde{g}_i(z^*,v).\]
		On the one hand, Lemma \ref{lem: Delta} gives $\tilde{g}_i(\beta(u),v)\leq \tilde{g}_i(\beta(u),u) + \|\tilde{g}_i(\beta(u),\cdot)\|_L \|u-v\| \leq 4 \|u-v\|$. On the other hand,
		\[ \tilde{g}_i(z^*,v) \leq \tilde{g}_i(z^*,u) + \|\tilde{g}_i(z^*,\cdot)\|_L \|u-v\| \leq -\frac{\varepsilon}{60}\|u-u_{NN}\| + 4\|u-v\|,\]
		where the second inequality follows from Lemma \ref{lem: Delta} and Lemma \ref{lem: z*}. Consequently,
		\[ \tilde{g}_i(z_v,v) \leq (1-\delta_v) 4\|u-v\| + \delta_v\left (4\|u-v\|-\frac{\varepsilon}{60}\|u-u_{NN}\|\right )= 4\|u-v\| - \frac{\varepsilon}{60}\delta_v\|u-u_{NN}\|=0.\qedhere\]
	\end{proof}
	
	Note that Lemma~\ref{lem: z in F_v} implies $\|\beta(v)\| \leq \|z_v\|$ thanks to the feasibility of $z_u$ and the optimality of $\beta(v)$. Next, we show that when $v$ and $u$ are close enough, $\|z_v\|$ is not much greater than $\|\beta(v)\|$ . This allows us to control $\big{|}\|\beta(v)\|-\|z_v\|\big{|}$, from which we will deduce a bound for $\|\beta(v)-z_v\|$. 
	
	\begin{lemma}\label{lem: z_v small} Assume that $\beta(u)\neq 0$ and $\|u-v\| \leq \frac{\varepsilon \|\beta(u)\|}{480}$. Then
		$$\|z_v\| \leq \left (1+4\delta_v \frac{\|u-u_{NN}\|}{\|\beta(u)\|}\right )\|\beta(v)\|.$$
	\end{lemma}
	
	\begin{proof}
		For convenience, write $a=\delta_v \frac{\|u-u_{NN}\|}{\|\beta(u)\|}$ and observe that $a\leq \frac{1}{2}$. Next, notice that
		\[ \|z_v\| \leq \|\beta(u)\|+ \delta_v \|z^*\| \leq \|\beta(u)\| + \delta_v \|u-u_{NN}\| = (1+a)\|\beta(u)\|.\]
		We claim that $\|\beta(u)\| \leq (1-a)^{-1} \|\beta(v)\|$, from which the result follows since then
		\[ \|z_v\| \leq (1+a)\|\beta(u)\| \leq \frac{1+a}{1-a} \|\beta(v)\| \leq (1+4a)\|\beta(v)\|.\]
		To prove the claim, consider $z^*_v = \beta(v) + \delta_v z^*$ and note that for any $i=1,\ldots,2\ell$ we have
		\[ g_i(z^*_v,u) = g_i(\beta(v),u) + \delta_v g_i ( z^*,u) \leq g_i(\beta(v),v)+ 4\|u-v\| - \delta_v \frac{\varepsilon}{60} \|u-u_{NN}\| = 0,\]
		where we have used that $g_i(\beta(v),v)\leq 0$, Lemma \ref{lem: Delta} to bound $\|g_i(\beta(v),\dot)\|_L\leq 4$, and Lemma~\ref{lem: z*} to estimate $g_i(z^*,u)$. Therefore, $z^*_v$ belongs to the feasible set $F_u$. Since $\beta(u)$ is the point in $F_u$ with minimal norm, we obtain that 
		\[ \|\beta(u)\|\leq \|z^*_v\| \leq \|\beta(v)\| + \delta_v \|z^*\| \leq \|\beta(v)\| + \delta_v \|u-u_{NN}\|.\]
		Consequently, 
		\[ \|\beta(v)\| \geq \|\beta(u)\| - \delta_v \|u-u_{NN}\| = (1-a)\|\beta(u)\|.\qedhere\]
	\end{proof}
	
	The following result will allow us to control the distance between $z_v$ and $\beta(v)$, using the bound for $\big{|}\|\beta(v)\|-\|z_v\|\big{|}$ provided by Lemma \ref{lem: z_v small}.
	
	\begin{lemma}\label{lem: circle}
		Given $\rho>0$, let $B \subseteq \mathbb{R}^m$ denote the ball centered at $0$ with radius $\rho$. Take $v \in B$ and consider the hyperplane $K_v=\{z \in \mathbb{R}^m \colon \langle v-z,v \rangle \leq 0\}$. Then
		\begin{equation}\label{eq: circle} 
			\max_{z \in K_v\cap B} \|v-z\|= \sqrt{\rho^2 - \|v\|^2}.
		\end{equation}
	\end{lemma}
	
	\begin{proof}
		If $v=0$ the statement is trivial. Hence, we may and do assume that $v\neq 0$. We use the notation $v=(v_1,\ldots,v_m)$ for the coordinates of $v$ in the standard basis. Up to rotations, we may assume that $v=(\|v\|,0,\ldots,0)$.\footnote{Or, equivalently, we may work in an orthonormal basis containing $v / \|v\|$ as the first basis element.} Let $w$ be a point in $K_v \cap B$ where the maximum in (\ref{eq: circle}) is attained. First, we claim that $\|w\|=\rho$. In fact, if $\|w\|<\rho$ then there is $\delta>0$ so that $w+\delta v \in K_v$. Indeed,
		\[ \langle v-(w+\delta v),v\rangle = \langle v-w,v\rangle - \delta \|v\|^2 \leq -\delta \|v\|^2 \leq 0.\]
		Moreover, if $\delta \leq 1-\frac{\|w\|}{\rho}$ then $\|w+\delta v\|\leq \|w\|+\delta \|v\| \leq\rho$. Now, observe that
		\begin{align*}
			\|v-(w+\delta v)\|^2 &= \langle v-(w+\delta v) , v-(w+\delta v) \rangle = \langle (1-\delta)v - w , (1-\delta)v - w \rangle \\
			&= (1-\delta)^2 \langle v , v \rangle - 2(1-\delta) \langle v , w \rangle  + \langle w , w \rangle  \\
			&= \langle v , v \rangle  - 2 \langle v , w \rangle  + \langle w , w \rangle  +(-2\delta+\delta^2)\langle v , v \rangle  + 2\delta \langle v , w \rangle \\
			&= \|v-w\|^2 - 2\delta \langle v-w , v \rangle  + \delta^2 \langle v , v \rangle.
		\end{align*}
		Since $\langle v-w , v \rangle \leq 0$, we conclude that $\|v- (w+\delta v)\|^2 > \|v-w\|^2$, contradicting the assumption that the maximum in (\ref{eq: circle}) is attained at $w$.  Hence, $\|w\|=\rho$ must hold as claimed.
		
		Next, notice that if $R$ is any rotation for which $R(v)=v$, then we have $\|v-R(w)\|=\|v-w\|$ and $\langle v-R(w),v \rangle = \langle v-w,v \rangle$. Therefore, we can effectively work in $\mathbb{R}^2$ by rotating $w - \langle w,v \rangle v$ into the second standard basis direction while fixing $v$. As a result, we may assume that $v=(\|v\|,0)$ and $w=(\rho\cos \theta, \rho\sin \theta)$ for some $\theta \in (-\pi,\pi]$ (since the rest of their coordinates may be assumed to be zero).  Since $\langle v-w, v\rangle \leq 0$ we must have $\rho\cos \theta \geq \|v\|$. In particular, $|\theta|<\pi/2$. Observe that
		\begin{align*}
			\|v-w\|^2&= \|(\rho\cos \theta -\|v\|, \rho\sin \theta)\|^2=\rho^2\cos^2 \theta -2\rho\|v\|\cos \theta + \|v\|^2 + \rho^2\sin^2\theta \\
			&= \rho^2-2\rho\|v\|\cos \theta +\|v\|^2.
		\end{align*}
		Therefore, the distance increases as $|\theta|$ increases, and so the maximum is attained at $w=(\|v\|, \sqrt{\rho^2 -\|v\|^2})$.
	\end{proof}
	
	We are finally ready to estimate the distance between $z_v$ and $\beta(v)$.
	
	\begin{lemma}\label{lem: zv-v'}
		Assume that $\beta(u)\neq 0$ and $\|u-v\| \leq \frac{\varepsilon \|\beta(u)\|}{480}$. Then
		$$\|z_v-\beta(v)\| \leq  \frac{128 \|u-u_{NN}\|^{1/2}}{\varepsilon^{1/2}} \|u-v\|^{1/2}.$$
	\end{lemma}
	
	\begin{proof}
		Recall that $a=\delta_v \frac{\|u-u_{NN}\|}{\|\beta(u)\|}$, let $\rho=(1+4a)\|\beta(v)\|$, and let $B=B(0,\rho)\subseteq \mathbb{R}^m$. Let $F_v$ be the feasible set for $v$, that is,
		\[F_v=\{z\in \mathbb{R}^m \colon g_i(z,v)\leq 0 \quad \forall \, i =1,\ldots,2\ell\}.\]
		Define the hyperplane $K_{\beta(v)}= \{z \in \mathbb{R}^m \colon \langle \beta(v)-z,\beta(v)\rangle \leq 0\}$. Since $\beta(v)$ is the orthogonal projection of $0$ onto the convex set $F_v$, we have that $\langle \beta(v)-z,\beta(v)\rangle \leq 0$ for any $z \in F_v$. In other words, $F_v\subseteq K_{\beta(v)}$. Thus, Lemma \ref{lem: z in F_v} gives $z_v \in K_{\beta(v)}$. Also, Lemma \ref{lem: z_v small} implies $z_v \in B$. Consequently, $z_v \in B\cap K_{\beta(v)}$ and so Lemma \ref{lem: circle}, combined with the fact that $a\leq \frac12$, gives
		\[ \|\beta(v)-z_v\| \leq \|\beta(v)\|\sqrt{(1+4a)^2 -1} = \|\beta(v)\|\sqrt{8a + 16a^2}  \leq \|\beta(v)\|\sqrt{16a} = 4\sqrt{a} \|\beta(v)\|. \]
		Next, recall that Lemma \ref{lem: z in F_v} implies that $\|\beta(v)\| \leq \|z_v\| \leq (1+a)\|\beta(u)\|\leq 2\|\beta(u)\|$. Therefore,
		\begin{align*}
			\|\beta(v)-z_v\| \leq 4\sqrt{a}\|\beta(v)\| &\leq 8\sqrt{a} \|\beta(u)\| \leq  8\sqrt{\delta_v \|u-u_{NN}\|\|\beta(u)\|} \leq 8 \left (\frac{240\|u-u_{NN}\|}{\varepsilon }\right )^{1/2} \|u-v\|^{1/2}\\
			&\leq \frac{124 \|u-u_{NN}\|^{1/2}}{\varepsilon^{1/2}} \|u-v\|^{1/2}.\qedhere
		\end{align*}
	\end{proof}
	
	We now have everything we need to prove Lemma \ref{lem: holder}.
	
	\begin{proof}[Proof of Lemma \ref{lem: holder}]
		Recall that (\ref{eq: trivial case}) shows that without loss of generality we may assume $u\neq u_{NN}$. Consider $r_u$ and $\widetilde{U}$ as defined in (\ref{eq: r}) and let $v \in \widetilde{U}$ with $v\neq u$. Consider $\delta_v$ and $z_v$ defined in (\ref{eq: delta}). First, we study the case when
		\[ \|u-v\| \geq \frac{\varepsilon \|\beta(u)\|}{480},\]
		which implies that $\|\beta(u)\| \leq 480\varepsilon^{-1} \|u-v\|$. Recall that Lemma \ref{lem: z in F_v} shows that $z_v$ belongs to the feasible set $F_v$, whence $\|\beta(v)\|\leq \|z_v\|$. Therefore,
		\begin{align*} \|\beta(u)-\beta(v)\|&\leq \|\beta(u)\|+\|\beta(v)\| \leq \|\beta(u)\|+\|z_v\| \leq 2\|\beta(u)\| + \delta_v\|z^*\| \leq 2\|\beta(u)\| + \delta_v \|u-u_{NN}\| \\
			&\leq \frac{960\|u-v\|}{\varepsilon } + \frac{240 \|u-v\|}{\varepsilon} = \frac{1200}{\varepsilon} \|u-v\| \leq \frac{1200}{\varepsilon} \|u-v\|^{1/2}.
		\end{align*}
		Next, we assume that 
		\begin{equation*}
			\|u-v\| \leq \frac{\varepsilon \|\beta(u)\|}{480}.
		\end{equation*}
		In particular, $\beta(u)\neq 0$.
		By the triangle inequality we have
		\[\|\beta(u)-\beta(v)\|\leq \|\beta(u)-z_v\|+ \|z_v-\beta(v)\|. \]
		Recall that $\|\beta(u)-z_v\|$ was already bounded in (\ref{eq: |u'-zv|}). Moreover, $\|z_v-\beta(v)\|$ was bounded by Lemma \ref{lem: zv-v'}. We conclude that
		\[ \|\beta(u)-\beta(v)\| \leq \left (\frac{480}{\varepsilon} + \frac{128 \|u-u_{NN}\|^{1/2}}{\varepsilon^{1/2}} \right ) \|u-v\|^{1/2}.\qedhere\]
	\end{proof}
	
	\subsection{Proof of Theorem \ref{theo: main infinite}}
	\label{ss:proof_main_infinite}
	We start by proving Theorem \ref{theo: u' holder}. The strategy will be to apply the special case $w=u$ proved in Lemma~\ref{lem: holder}, together with the fact that the constants $\widetilde{C}_u$ obtained in such result behave in a Lipschitz manner.
	
	\begin{proof}[Proof Theorem \ref{theo: u' holder}]
		We divide the proof in two cases. Namely, $u \in \widetilde{X}\setminus \overline{X}$ and $u \in \overline{X}$.
		
		\paragraph{Case 1: $u \in \widetilde{X}\setminus \overline{X}$.} 
		Write $\widetilde{U}$ and $\widetilde{C}_u$ for the neighborhood of $u$ and the constant provided by Lemma \ref{lem: holder}. Thus, for any $v \in \widetilde{U}$ we have
		\[ \|\beta(v)-\beta(u)\|\leq \widetilde{C}_u \|v-u\|^{1/2}.\]
		As we discussed in Remark \ref{rem: constants for lemma}, for $u \notin \overline{X}$ we can take 
		\begin{equation}\label{eq: radius} r_u =\min\left\{1,\frac{\varepsilon \|u-u_{NN}\|}{480}\right\}, \quad \quad  \widetilde{U}=B(u,r_u) \cap \widetilde{X}, \quad \quad \widetilde{C}_u= \frac{1}{\varepsilon}\max\left \{1200, 480+128\|u-u_{NN}\|^{1/2}\right \}.
		\end{equation}
		Take $v$, $w \in \widetilde{U}$. We now distinguish two subcases. First, assume that 
		\[\|v-w\| \geq \frac{1}{3} \big{(}\|v-u\| + \|u-w\|\big{)}.\] 
		In this case, we have
        \begin{align*}\|\beta(v)-\beta(w)\|^2 &\leq \big{(}\|\beta(v)-\beta(u)\| + \|\beta(u)-\beta(w)\|\big{)}^2 \leq 2\big{(}\|\beta(v)-\beta(u)\|^2 + \|\beta(u)-\beta(w)\|^2\big{)}\\ &\leq 2\widetilde{C}_u^2\big{(}\|v-u\|+\|u-w\|\big{)} \leq 6\widetilde{C}_u^2\|v-w\|.
    \end{align*}
		Consequently,
		\[\|\beta(v)-\beta(w)\| \leq \sqrt{6} \widetilde{C}_u \|v-w\|^{1/2}.\]
		
		Next, assume that
		\[\|v-w\| < \frac{1}{3} \big{(}\|v-u\| + \|u-w\|\big{)}.\]
		In particular, we have $\|v-w\|< \frac{2}{3} r_u.$ We claim that $r_v\geq \frac{2}{3}r_u$ and $\widetilde{C}_v \leq \sqrt{\frac{3}{2}} \widetilde{C}_u$. If the claim holds, since $\|w-v\| < \frac{2}{3} r_u \leq r_v$, then Lemma \ref{lem: holder} applied to $v$ yields
		\[ \|\beta(v)-\beta(w)\| \leq \widetilde{C}_v\|v-w\|^{1/2} \leq \sqrt{\frac{3}{2}}\widetilde{C}_u\|v-w\|^{1/2}.\] 
		To prove the claim, notice that from (\ref{eq: u-u_NN Lipschitz}) and $\|u-v\|<r_u$ we obtain
		\[ \left |\|v-v_{NN}\| - \|u-u_{NN}\|\right | \leq \|(v-v_{NN}) - (u-u_{NN})\| \leq 3\|u-v\| \leq \frac{1}{160} \|u-u_{NN}\|. \]
		One can easily deduce from the above inequality that
		\begin{equation}\label{eq: u-u_NN vs v-v_NN}
			\frac{2}{3} \|u-u_{NN}\| \leq \|v-v_{NN}\| \leq \frac{3}{2} \|u-u_{NN}\|.
		\end{equation}
		Observe that, by definition of $r_u$, we have $\|u-v\|< \|u-u_{NN}\|$, which implies that $v \notin \overline{X}$. Therefore, $r_u$, $r_v$, $\widetilde{C}_u$, and $\widetilde{C}_v$ are given by (\ref{eq: radius}). Hence, we can use (\ref{eq: u-u_NN vs v-v_NN}) to prove the claim. Indeed, if $r_v =1$ then clearly $r_v \geq r_u$. Otherwise,
		\[ r_u=\min\left\{1,\frac{\varepsilon \|u-u_{NN}\|}{480}\right\} \leq \frac{\varepsilon \|u-u_{NN}\|}{480} \leq \frac{3}{2} \frac{\varepsilon \|v-v_{NN}\|}{480} =\frac{3}{2}r_v. \]
		Finally, note that
		\begin{align*} \widetilde{C}_v&= \frac{1}{\varepsilon}\max\left \{1200, 480+128\|v-v_{NN}\|^{1/2}\right \} \\
			&\leq \frac{1}{\varepsilon}\max\left \{1200, 480+128\sqrt{\frac{3}{2}}\|u-u_{NN}\|^{1/2}\right \} \\
			&\leq \sqrt{\frac{3}{2}} \cdot \frac{1}{\varepsilon}\max\left \{1200, 480+128\|u-u_{NN}\|^{1/2}\right \}= \sqrt{\frac{3}{2}} \widetilde{C}_u.
		\end{align*}
		
		\paragraph{Case 2: $u \in  \overline{X}$.} 
		In this case, we have $u=u_{NN}$. Let $U'=B(u,\frac{1}{2})\cap \widetilde{X}$ and fix $v$, $w \in U'$ with $w\neq v$. First, assume that
		\[ \|w-v\|< \frac{\varepsilon \|v-v_{NN}\|}{480}.\]
		In particular, $v\notin \overline{X}$. Therefore, Remark \ref{rem: constants for lemma} provides an explicit expression for the radius $r_v$ provided by Lemma \ref{lem: holder} when applied to $v$. In this case, $\|w-v\|\leq 1$ yields $\|w-v\|<r_v$ and Lemma \ref{lem: holder} gives
		\[ \|\beta(w)-\beta(v)\|\leq \widetilde{C}_v\|v-w\|^{1/2},\]
		where $\widetilde{C}_v$ is a constant also explicitly described in Remark \ref{rem: constants for lemma}. Note that $u=u_{NN}$ and (\ref{eq: u-u_NN Lipschitz}) imply $\|v-v_{NN}\|\leq 3\|u-v\| \leq 3/2$. Therefore $\widetilde{C}_v$ satisfies
		\[\widetilde{C}_v=\frac{1}{\varepsilon} \max\{ 1200, 480+128\|v-v_{NN}\|^{1/2}\}\leq \frac{1200}{\varepsilon}.\]
		Finally, let us study the case when
		\[ \|w-v\|\geq \frac{\varepsilon \|v-v_{NN}\|}{480}.\]
		Under this assumption, the triangle inequality, (\ref{eq: constraint_norm}), and (\ref{eq: u-u_NN Lipschitz}) give
		\begin{align*}
			\|\beta(w)-\beta(v)\|&\leq \|\beta(w)\|+\|\beta(v)\| \leq \|w-w_{NN}\|+ \|v-v_{NN}\|= \|w-w_{NN}\| - \|v-v_{NN}\| +2\|v-v_{NN}\|\\
			&\leq \|(w-w_{NN})-(v-v_{NN})\| + 2\|v-v_{NN}\| \leq 3\|w-v\| + 2\|v-v_{NN}\|\\
			&\leq \left (\frac{960}{\varepsilon} + 3\right ) \|w-v\| \leq  \frac{963}{\varepsilon} \|w-v\|^{1/2}.\qedhere
		\end{align*}
	\end{proof}
	
	\begin{remark}\label{rem: constants theo holder}
		The proof of Theorem \ref{theo: u' holder} provides explicit expressions for the neighborhood $U'$ and the constant $C'_u$. If $u \in \overline{X}$, then one can take $U'=B(u,1/2)\cap \widetilde{X}$, and $C'_u=1200/\varepsilon$. Let $\widetilde{U}$ and $\widetilde{C}_u$ be the neighborhood and constant discussed in Remark \ref{rem: constants for lemma}. For $u \notin \overline{X}$, one can take $U'=\widetilde{U}\cap\widetilde{X}$ and $C'_u=\sqrt{6}\widetilde{C}_u$.
	\end{remark}
	
	We are finally ready to prove Theorem \ref{theo: main infinite}.
	
	\begin{proof}[Proof of Theorem \ref{theo: main infinite}]
		Given $\varepsilon \in (0,1)$, consider the terminal embedding $f\colon \widetilde{X} \longrightarrow \mathbb{R}^m$ with distortion $\varepsilon$ provided by Lemma \ref{lem: is terminal}. Recall that $f$ is defined as
		\[f(u)=\begin{cases}
			\big{(}\Pi u_{NN} +u',\sqrt{\|u-u_{NN}\|^2 - \|u'\|^2} \big{)} & \mbox{if }  u \in \widetilde{X} \setminus\{\overline{X}\};\\
			\big{(}\Pi u,0\big{)} & \mbox{if } u \in \overline{X},
		\end{cases}\]
		where $\Pi \in \mathbb{R}^{d \times m}$ provides $\frac{\varepsilon}{60}$-convex hull distortion for $S_X$, $u_{NN}$ is the closest point from $\overline{X}$ to $u$, and $u'$ is the solution of the optimization problem \ref{optimization problem infinite}. Let $p\colon \widetilde{X} \longrightarrow \mathbb{R}$ defined by $p(u)=\|u-u_{NN}\|^2 - \|u'\|^2$. Then, for any points $v$, $w$ in $\widetilde{X}$ we have that
		\begin{align*}
			\|f(v)-f(w)\|^2&= \|\Pi v_{NN} - \Pi w_{NN} + \beta(v)-\beta(w)\|^2 + \left |\sqrt{p(v)}-\sqrt{p(w)}\right |^2 \\
			&\leq \big{(} \|\Pi v_{NN} - \Pi w_{NN}\|+ \|\beta(v)-\beta(w)\| \big{)}^2 + |p(v)-p(w)|\\
			&\leq 2\|\Pi v_{NN} - \Pi w_{NN}\|^2 + 2\|\beta(v)-\beta(w)\|^2 + |p(v)-p(w)|
		\end{align*}
		Using the embedding condition and (\ref{eq: u_NN Lipschitz}), we can bound
		\[\|\Pi v_{NN} - \Pi w_{NN}\| \leq (1+\varepsilon) \|v_{NN}-w_{NN}\| \leq 4 \|v-w\|.\]
		Moreover, Theorem \ref{theo: u' holder} shows that there are a radius $0<r_u<\frac{1}{2}$ and a constant $C'_u>0$ such that $\|\beta(v)-\beta(w)\| \leq C'_u \|v-w\|^{1/2}$ for any $v$, $w \in B(u,r_u)\cap \widetilde{X}$. Thus, it remains to bound $|p(v)-p(w)|$. First, by the triangle inequality we have
		\begin{align*}\big{|}\|\beta(v)\|^2 - \|\beta(w)\|^2\big{|}&=\big{(}\|\beta(v)\|+\|\beta(w)\|\big{)}\big{|}\|\beta(v)\|-\|\beta(w)\|\big{|} \leq \big{(}\|\beta(v)\|+\|\beta(w)\|\big{)}\|\beta(v)-\beta(w)\|\\
			&\leq \big{(}\|v-v_{NN}\| +\|w-w_{NN}\|\big{)} \|\beta(v)-\beta(w)\|\\
			&\leq C'_u \big{(}\|v-v_{NN}\| +\|w-w_{NN}\|\big{)}\|v-w\|^{1/2}.
		\end{align*}
		Similarly, the triangle inequality and (\ref{eq: u-u_NN Lipschitz}) give
		\begin{align*}
			\big{|}\|v-v_{NN}\|^2 -\|w-w_{NN}\|^2\big{|} &\leq \big{(}\|v-v_{NN}\| +\|w-w_{NN}\|\big{)}\|(v-v_{NN}) - (w-w_{NN})\|\\
			& \leq 3 \big{(}\|v-v_{NN}\| +\|w-w_{NN}\|\big{)} \|v-w\| \\
			&\leq 3 \big{(}\|v-v_{NN}\| +\|w-w_{NN}\|\big{)} \|v-w\|^{1/2}.
		\end{align*}
		Finally, notice that
		\[ \|v-v_{NN}\| \leq \|u-u_{NN}\| + \|(v-v_{NN})-(u-u_{NN})\| \leq \|u-u_{NN}\| + 3\|u-v\| \leq \|u-u_{NN}\|+3r_u,\]
		where the same inequality also holds for $w$. Consequently, we obtain
		\[ |p(v)-p(w)| \leq (3+C'_u)\big{(}\|v-v_{NN}\| +\|w-w_{NN}\|\big{)} \|v-w\|^{1/2} \leq (6+2C'_u)(\|u-u_{NN}\|+3r_u)\|v-w\|^{1/2}.\]
		We conclude that for any $v$, $w \in B(u,r_u)\cap \widetilde{X}$ we have
		\begin{align*}
			\|f(v)-f(w)\|^2 &\leq 2\|\Pi v_{NN} - \Pi w_{NN}\|^2 + 2\|\beta(v)-\beta(w)\|^2 + |p(v)-p(w)|\\
			&\leq 32\|v-w\|^2 + 2C'^{2}_u\|v-w\| + (6+2C'_u)(\|u-u_{NN}\|+3r_u)\|v-w\|^{1/2}\\
			&\leq \left (32+2C'^{2}_u + (6+2C'_u)(\|u-u_{NN}\| + 2) \right ) \|v-w\|^{1/2}.\qedhere
		\end{align*}
	\end{proof}
	
	\section*{Acknowledgements}

Simone Brugiapaglia acknowledges the support of the Natural Sciences and Engineering Research Council of Canada (NSERC) through grant RGPIN-2020-06766. Mark Iwen and Rafael Chiclana Vega were supported in part by NSF DMS 2106472. Tim Hoheisel was partially supported by a grant RGPIN-2024-04116 of the Natural Sciences and Engineering Research Council (NSERC) of Canada.  The authors are also grateful to the  {\em Centre de Recherches Math\'ematiques (CRM)}, in particular its Applied Mathematics Laboratory, for supporting Mark Iwen's visit in Summer 2022.


\begin{thebibliography}{00}
		
		\bibitem{A2001Database}
		D.~Achlioptas.
		\newblock Database-friendly random projections.
		\newblock In {\em Proceedings of the twentieth ACM SIGMOD-SIGACT-SIGART
			symposium on Principles of database systems}, pages 274--281. ACM, 2001.
		
		\bibitem{BaC11}
		Heinz~H. Bauschke and Patrick~L. Combettes.
		\newblock {\em Convex Analysis and Monotone Operator Theory in Hilbert Spaces}.
		\newblock Springer Publishing Company, Incorporated, 1st edition, 2011.
		
		\bibitem{Cherapanamjeri2022terminal}
		Y.~Cherapanamjeri and J.~Nelson.
		\newblock Terminal embeddings in sublinear time.
		\newblock In \emph{Proceedings of the 62nd IEEE Symposium on Foundations of Computer Science (FOCS)}, pages 1209--1216, 2021.
		
		\bibitem{chiclana2024on}
		Rafael Chiclana, Mark~A. Iwen, and Mark~Philip Roach.
		\newblock On outer bi-{L}ipschitz extensions of linear
		{J}ohnson-{L}indenstrauss embeddings of subsets of $\mathbb{R}^n$, 2024.
		
		\bibitem{DG2003Elementary}
		S.~Dasgupta and A.~Gupta.
		\newblock An elementary proof of a theorem of {J}ohnson and {L}indenstrauss.
		\newblock {\em Random Structures \& Algorithms}, 22(1):60--65, 2003.
		
		\bibitem{dirksen2016dimensionality}
		Sjoerd Dirksen.
		\newblock Dimensionality reduction with subgaussian matrices: a unified theory.
		\newblock {\em Found. Comput. Math.}, 16(5):1367--1396, 2016.
		
		\bibitem{elkin2017terminal}
		Michael Elkin, Arnold Filtser, and Ofer Neiman.
		\newblock Terminal embeddings.
		\newblock {\em Theoretical Computer Science}, 697:1--36, 2017.
		
		\bibitem{facchinei2003finite}
		Francisco Facchinei and Jong-Shi Pang.
		\newblock {\em Finite-dimensional variational inequalities and complementarity
			problems. {V}ol. {I}}.
		\newblock Springer Series in Operations Research. Springer-Verlag, New York,
		2003.
		
		\bibitem{federer1959curvature}
		Herbert Federer.
		\newblock Curvature measures.
		\newblock {\em Trans. Amer. Math. Soc.}, 93:418--491, 1959.
		
		\bibitem{FH2013Mathematical}
		Simon Foucart and Holger Rauhut.
		\newblock {\em A mathematical introduction to compressive sensing}.
		\newblock Springer, 2013.
		
		\bibitem{iwen2023lower}
		Mark Iwen, Benjamin Schmidt, and Arman Tavakoli.
		\newblock Lower bounds on the low-distortion embedding dimension of
		submanifolds of rn.
		\newblock {\em Applied and Computational Harmonic Analysis}, 65:170--180, 2023.
		
		\bibitem{iwen2024on}
		Mark~A. Iwen, Benjamin Schmidt, and Arman Tavakoli.
		\newblock On {F}ast {J}ohnson-{L}indenstrauss {E}mbeddings of {C}ompact
		{S}ubmanifolds of {$\Bbb R^N$} with {B}oundary.
		\newblock {\em Discrete Comput. Geom.}, 71(2):498--555, 2024.
		
		\bibitem{johnson1984extensions}
		William Johnson and Joram Lindenstrauss.
		\newblock Extensions of lipschitz maps into a hilbert space.
		\newblock {\em Contemporary Mathematics}, 26:189--206, 01 1984.
		
		\bibitem{larsen2017optimality}
		Kasper~Green Larsen and Jelani Nelson.
		\newblock Optimality of the {J}ohnson-{L}indenstrauss lemma.
		\newblock In {\em 58th {A}nnual {IEEE} {S}ymposium on {F}oundations of
			{C}omputer {S}cience---{FOCS} 2017}, pages 633--638. IEEE Computer Soc., Los
		Alamitos, CA, 2017.
		
		\bibitem{mahabadi2018nonlinear}
		Sepideh Mahabadi, Konstantin Makarychev, Yury Makarychev, and Ilya Razenshteyn.
		\newblock Nonlinear dimension reduction via outer bi-{L}ipschitz extensions.
		\newblock In {\em S{TOC}'18---{P}roceedings of the 50th {A}nnual {ACM} {SIGACT}
			{S}ymposium on {T}heory of {C}omputing}, pages 1088--1101. ACM, New York,
		2018.
		
		\bibitem{narayanan2019optimal}
		Shyam Narayanan and Jelani Nelson.
		\newblock Optimal terminal dimensionality reduction in {E}uclidean space.
		\newblock In {\em S{TOC}'19---{P}roceedings of the 51st {A}nnual {ACM} {SIGACT}
			{S}ymposium on {T}heory of {C}omputing}, pages 1064--1069. ACM, New York,
		2019.
		
		\bibitem{neumann1928zur}
		J.~v.~Neumann.
		\newblock Zur {T}heorie der {G}esellschaftsspiele.
		\newblock {\em Math. Ann.}, 100(1):295--320, 1928.
		
		\bibitem{vershynin2018high-dimensional}
		Roman Vershynin.
		\newblock {\em High-dimensional probability}, volume~47 of {\em Cambridge
			Series in Statistical and Probabilistic Mathematics}.
		\newblock Cambridge University Press, Cambridge, 2018.
		\newblock An introduction with applications in data science, With a foreword by
		Sara van de Geer.
		
	\end{thebibliography}
\end{document}